\documentclass[11pt,a4paper]{amsart}

\usepackage[utf8]{inputenc}
\usepackage[T1]{fontenc}
\usepackage[english]{babel}

\usepackage{amsmath,amscd,amssymb}
\usepackage{amsthm}

\usepackage{xcolor}

\usepackage[pdfencoding=auto]{hyperref}  
 \hypersetup{
     colorlinks,
     linkcolor= {blue!80!black},
     citecolor={blue!80!black},
     urlcolor={blue!80!black}}
     
\usepackage{mathrsfs}

\usepackage{parskip} 

\usepackage[left=3.3cm,right=3.3cm,top=4.5cm,bottom=4.5cm]{geometry}

\usepackage{enumerate}

\usepackage[nameinlink]{cleveref}

\usepackage{tikz-cd} 

\newtheorem{main-theorem}{Theorem}
\newtheorem{main-corollary}[main-theorem]{Corollary}
\newtheorem{main-algorithm}[main-theorem]{Algorithm}
\newtheorem{theorem}{Theorem}[section]

\newtheorem{lemma}[theorem]{Lemma}
\newtheorem{proposition}[theorem]{Proposition}
\newtheorem{definition}[theorem]{Definition}
\newtheorem{algorithm}[theorem]{Algorithm}
\theoremstyle{remark}
\newtheorem{remark}[theorem]{Remark}
\newtheorem*{acknowledgements}{Acknowledgements}

\numberwithin{equation}{section}

\newcommand{\mB}{\mathrm{B}} 
\newcommand{\IB}{\mathbb{B}}
\newcommand{\cB}{\mathcal{B}} 

\newcommand{\IC}{\mathbb{C}} 
\newcommand{\cC}{\mathcal{C}}
\newcommand{\cD}{\mathcal{D}}
\newcommand{\cE}{\mathcal{E}}

\newcommand{\cF}{\mathcal{F}}

\newcommand{\gH}{\mathfrak{H}} 

\newcommand{\cL}{\mathcal{L}}
\newcommand{\cM}{\mathcal{M}}
\newcommand{\N}{\mathbb{N}}

\newcommand{\R}{\mathbb{R}}

\newcommand{\IS}{\mathbb{S}}
\newcommand{\cS}{\mathcal{S}} 

\newcommand{\cV}{\mathcal{V}}

\newcommand{\la}{\lambda}

\newcommand{\SO}{\mathrm{SO}}

\newcommand{\norm}[1]{\lVert #1 \rVert} 
\newcommand{\normm}[1]{{\left\vert\kern-0.25ex\left\vert\kern-0.25ex\left\vert #1 \right\vert\kern-0.25ex\right\vert\kern-0.25ex\right\vert}}

 \newcommand{\br}[1]{\left \langle #1 \right \rangle} 
\newcommand{\ol}[1]{\overline{#1}}

\newcommand{\qe}{q^{\mathrm{B}}} 
\newcommand{\ve}{V^{\mathrm{B}}}

\newcommand{\rad}{\mathrm{rad}}

\DeclareMathOperator{\Spec}{\mathrm{Spec}}

\DeclareMathOperator{\Id}{Id}

\newcommand{\loc}{\mathrm{loc}}

\let\Re\relax
\DeclareMathOperator{\Re}{\mathrm{Re}}
\let\Im\relax
\DeclareMathOperator{\Im}{\mathrm{Im}}

\title[Stable factorization of the   Calderón problem]{Stable factorization of the   Calderón problem  via the Born approximation}

\author[T. Daudé]{Thierry Daudé}
\address[TD]{Laboratoire de math\'ematiques de Besan\c con, UMR CNRS 6623, Universit\'e de Franche-Comt\'e, 25030, Besan\c con, France.}
\email{thierry.daude@univ-fcomte.fr}

\author[F. Macià]{Fabricio Macià}
\address[FM]{M$^2$ASAI. Universidad Politécnica de Madrid, ETSI Navales, Avda. de la Memoria, 4, 28040, Madrid, Spain.}
\email{fabricio.macia@upm.es}

\author[C. Meroño]{Cristóbal Meroño}
\address[CM]{M$^2$ASAI. Universidad Politécnica de Madrid, ETSI Navales, Avda. de la Memoria, 4, 28040, Madrid, Spain.}
\email{cj.merono@upm.es}

\author[F. Nicoleau]{François Nicoleau}
\address[FN]{Laboratoire de Math\'ematiques Jean Leray, UMR CNRS 6629, 2 Rue de la Houssini\`ere BP 92208, F-44322 Nantes Cedex 03.}
\email{francois.nicoleau@univ-nantes.fr}

\begin{document}

\begin{abstract}
   In this article, we prove the existence of the Born approximation in the context of the radial Calderón problem for Schrödinger operators. 
   The Born approximation naturally appears as the linear component of a factorization of the Calderón problem; we show that the non-linear part, obtaining the potential from the Born approximation, enjoys several interesting properties. First, this map is local, in the sense that knowledge of the Born approximation in a neighborhood of the boundary is equivalent to knowledge of the potential in the same neighborhood, and, second, it is Hölder stable. This proves that the ill-posedness of the Calderón problem arises from the linear step, which consists in computing the Born approximation from the DtN map by solving a Hausdorff moment problem. Moreover, we present an effective algorithm to compute the potential from the Born approximation. Finally, we use the Born approximation to obtain a partial characterization of the set of DtN maps for radial potentials. The proofs of these results do not make use of Complex Geometric Optics solutions or their analogs; they are based on results from inverse spectral theory for Schrödinger operators on the half-line, in particular on the concept of $A$-amplitude introduced by Barry Simon. \\
   
\end{abstract}

\maketitle

\section{Introduction}

\subsection{The problem and the setting}
Let $\Omega\subset\R^d$, $d\geq 2$, be a smooth bounded domain, and denote by $\partial_\nu$ the outward normal unit vector field on $\partial\Omega$. The Calderón problem in $\Omega$ is the inverse problem of reconstructing a positive conductivity function $\gamma$ in the equation
\begin{equation*}  
\left\{
\begin{array}{rll}
\nabla \cdot(\gamma \nabla u)  =& 0   &\text{ in }   \Omega\subset \R^d,\\
u  =& f, &\text{ on }   \partial \Omega,
\end{array}\right.
\end{equation*}
from the knowledge of the Dirichlet to Neumann map (in what follows, the DtN map)
\begin{equation*}
    \widetilde{\Lambda}_\gamma : H^{1/2}(\partial\Omega) \ni f \longmapsto \gamma \partial_\nu u|_{\partial \Omega}\in H^{-1/2}(\partial\Omega) .
\end{equation*}
This problem goes back to Calderón, who considered it since the fifties and    published his results in 1980  in \cite{calderon}.\footnote{This reference has been reprinted in \cite{calderon_rep}.} It is well known  that this inverse problem is closely related to the problem of reconstructing a real-valued potential $V$ in the Schrödinger equation
\begin{equation}  \label{id:calderon_direct}
\left\{
\begin{array}{rll}
-\Delta v +Vv =&0    &\text{ in }   \Omega,\\
v  =&g, &\text{ on }   \partial \Omega,
\end{array}\right.
\end{equation}
from the corresponding DtN map (provided it is well-defined)
\begin{equation*}
    \Lambda_V :   H^{1/2}(\partial\Omega) \ni g \longmapsto  \partial_\nu v|_{\partial \Omega} \in H^{-1/2}(\partial\Omega) .
\end{equation*}
In fact, in the case of  regular conductivities, the conductivity problem can be reduced to the Schrödinger problem by the Liouville transform $v= \sqrt{\gamma}u$, $g = \sqrt{\gamma}f$, and with the potential $V = \frac{\Delta {\sqrt{\gamma}}}{\sqrt{\gamma}}$ (see \cite{GU92} for instance).

One can interpret the Calderón problem for the Schrödinger equation \eqref{id:calderon_direct}, also known as the Gel'fand-Calderón problem \cite{Gelfand_54},  as the problem of inverting the non-linear map
\begin{equation*}
    \Phi:  V \longmapsto \Lambda_V -\Lambda_0,
\end{equation*}
where the potential $V$ lies in $\cV$, a class of real-valued potentials defined on $\Omega$, and $\Lambda_0$ is the DtN map associated with the free Laplacian ($V=0$).

 The uniqueness question for the Calderón problem, which amounts to showing the injectivity of  $\Phi$, has received a great amount of attention since the first results of Kohn and Vogelius \cite{KoVo84} and Sylvester and Uhlmann \cite{SU87}, see \Cref{sec:reconstr_intro} for more references.
 The reader can also consult the book \cite{FSUBook} for a comprehensive guide to the literature on the Calderón Problem and related inverse problems.

The inverse map $\Phi^{-1}$ is generally never globally continuous, as shown by Alessandrini \cite{Alessandrini88} in $L^\infty$ (for $L^p$ results, see \cite{Ale_Cabib_08,Faraco_Kurylev_Ruiz_14}), so the Calderón problem is an ill-posed inverse problem.
In spite of this, $\Phi^{-1}$  is conditionally stable. This means that $\Phi$ is a homeomorphism when restricted to compact subsets of potentials $K$ and that $\Phi^{-1}$ has a modulus of continuity $\omega_K$ on $\Phi(K)$.

The stability of the reconstruction of $\gamma$ or $V$ from the DtN map is the question of estimating $\omega_K$. For example, Alessandrini proved in \cite{Alessandrini88} for $d\geq 3$ that  $\omega_K$ is logarithmic  when $K \subset L^\infty$ is a ball in a certain Sobolev space, to provide the required compactness (in $d=2$, see for example, \cite{Barcelo_Barcelo_Ruiz_01}).
This result is sharp, as shown by Mandache in the case of the unit ball \cite{Mand00} or \cite{KRS21} in more general settings. In conclusion, even assuming \textit{a priori} regularity on potentials, the problem of determining the potential $V$ from the DtN map $\Lambda_V$ is still highly unstable.

In this work, we show that when $\Omega = \IB^d :=\{x\in \R^d : |x|<1 \}$ is the unit ball in Euclidean space, with $\partial\Omega=\IS^{d-1}$, $d\geq2$, and $\cV$ being a general class of real-valued radial potentials, one can factor the inverse map $\Phi^{-1}$ as:
\begin{equation}  \label{eq:diagram_1}
\begin{tikzcd}
\cM:=\Phi(\cV) \arrow[dr,"d\Phi_0^{-1}"] \arrow[rr,"\Phi^{-1}"] && \cV \\
 &  \cB \arrow[ur,"\Phi_\mB^{-1}"]
\end{tikzcd}
\qquad \qquad
\begin{tikzcd}
\Lambda_V-\Lambda_0 \; \ar[dr,|->,"d\Phi_0^{-1}"] \ar[rr,|->,"\Phi^{-1}"] && V \\
 &  \ve \ar[ur,|->,"\Phi_\mB^{-1}"]
\end{tikzcd}
\end{equation}
The map $d\Phi_0$ is the Fréchet differential of $\Phi$ at  $V=0$, which is injective, as proved  by \cite{calderon} in the conductivity problem. 

One of the main contributions of this article is showing that $d\Phi_0^{-1}$ can be extended to $\cM:=\Phi(\cV)$, the set of all operators  $\Lambda_V - \Lambda_0$  arising from potentials in $\cV$, and that it maps each $\Lambda_V - \Lambda_0$ to a function $\ve$ that is  supported on $\IB^d$. In other words, we show that there exists a function $\ve$  satisfying
\begin{equation} \label{id:born_aprox_def}
    d\Phi_{0}(\ve) =  \Lambda_V-\Lambda_0.
\end{equation}
This is a non-trivial fact, since there is no reason \textit{a priori} for $\Lambda_V-\Lambda_0$ to belong to the range of $d\Phi_{0}$, and it is the key to rigorously factorizing the problem as in \eqref{eq:diagram_1}. This is accomplished in \Cref{mt:existence}. As far as we know, the existence of this object has been postulated in the more numerically-oriented literature (see, for instance, \cite{HaSe2010}), but no rigorous proof of its existence has been provided. 

The function $\ve$ plays an analogous role to the Born approximation in scattering problems. Therefore, we will refer to $\ve$ as the Born approximation of the potential $V$. It turns out that $\ve$ is well defined even if  $\Lambda_V$ is replaced by the corresponding Cauchy data, a necessary technical step when \eqref{id:calderon_direct} has no unique solution.
Also, in \Cref{mt:existence} we prove that $\ve$ is the solution of a certain moment problem and that its moments are the eigenvalues of $\Lambda_V - \Lambda_0$. This has the remarkable consequence of implying a partial characterization result for the DtN maps of radial potentials, see \Cref{remark:characterization} and \Cref{sec:appendix_moments} for more details. We will also prove that $V-\ve$ is always a continuous function outside the origin; this property of recovery of singularities is stated in \Cref{mt:approximation}.

We now turn our attention to the map $\Phi_\mB$ in \eqref{eq:diagram_1}.
The map $\Phi_\mB :V \mapsto \ve $ is a non-linear bijection from $\cV$ onto $\cB:=d\Phi_0^{-1}(\cM)$. In particular, we will show that by inverting $\Phi_\mB$, one can determine $V$ directly from $\ve$ with Hölder stability in a $L^1$ weighted space — see Theorems 
\ref{mt:uniqueness}, \ref{mt:stability_local}, and \ref{mt:stability_global}. This stability estimate does not require compactness, only mild conditionality in $L^p$. More graphically:
\begin{equation} \label{eq:diagram_2}
\begin{tikzcd}[column sep=10em,row sep=10em]
\Lambda_V -\Lambda_0 \ar[r,|->, "\text{Conditional log stability}", "\text{Linear bijection}"']  & \ve \ar[r,|->, "\text{Hölder stability}", "\text{Non-linear}"'] & V.
\end{tikzcd}
\end{equation}
This shows that the linear part of the factorization is responsible for the instability and the ill-posedness of the radial Calderón problem. 
Exploiting this factorization, we provide an explicit reconstruction algorithm, \Cref{alg:rec}, to reconstruct $V$ from $\Lambda_V$. 

The proofs of these results do not involve the use of exponentially growing solutions of \eqref{id:calderon_direct}; instead, they are based on the approach to 1-d inverse spectral theory developed by Simon in \cite{IST1}.
These results have been carried over to the Calderón Problem for conductivities
\cite{radial_cond} and for the fixed energy Schrödinger equation \cite{MMS-M25}.
 
 \subsection{Existence of the Born approximation} 
The main difficulty in proving that \eqref{id:born_aprox_def} has a solution $\ve$ lies in showing that the map $d\Phi_{0}^{-1}$ can be extended from its natural domain—the ``tangent space'' to $\cM$, following the analogy from differential geometry—to the whole set $\cM$ of operators $\Lambda_V-\Lambda_0$. To do this, we will  reformulate \eqref{id:born_aprox_def} in terms of a moment problem, as we now show.

Assume for now  that $V\in L^\infty_\rad(\IB^d;\R)$ \footnote{We use the notation $L^p_\rad(\IB^d)$ for the closed subspace of radial functions in $L^p(\IB^d)$.} and that $\Lambda_V$ is well-defined. In that case, the spectral theory of $\Lambda_V$ is easy to describe. Denote the subspace of spherical harmonics of degree $k \in \N_0$ by $\gH_k$ (here $\N_0$ stands for the set of non-negative integers).
 The spaces $\gH_k$ are mutually orthogonal in $L^2(\IS^{d-1})$, and by separation of variables, one can show that 
\begin{equation} \label{id:DtN_map_eigenvalue}
    \Lambda_V|_{\gH_k}=\lambda_k[V] \Id_{\gH_k},
\end{equation} 
for every $k\in\N_0$. This shows that each $Y_k \in \gH_k$ is an eigenfunction of the DtN map. 
For example, when $V=0$, a direct computation gives that $\Lambda_0(Y_k)=k Y_k$, so that $\lambda_k[0] = k$ for every $k\in\N_0$.

For a radial function $F = F_0(|\cdot|)$ and any $k\in\N_0$, we define the moments:
\begin{equation} \label{id:moments_ball}
    \sigma_k[F] := \frac{1}{|\IS^{d-1}|} \int_{\IB^d} F(x) |x|^{2k}\, dx  = \int_0^1 F_0(r) r^{2k+d-1} \,dr.
\end{equation}
Let $d\Phi_0(V)$ be the Fréchet derivative of $\Phi$ at the zero potential in the direction  $V \in L^\infty_\rad(\IB^d)$. One can prove, see \Cref{sec:subsec_technical_calderon}, that $d\Phi_0(V)$ is the bounded operator in $L^2(\IS^{d-1})$ satisfying
\begin{equation} \label{e:frechet_dif}
 d\Phi_0(V)|_{\gH_k}=\sigma_k[V]\Id_{\gH_k}, \qquad \forall k\in\N_0. 
\end{equation}
This connection between the \textit{linearized Calderón problem}, recovering $V$ from the knowledge of $d\Phi_0(V)|_{\gH_k}$, and the moment problem  has been known for a long time; see, for instance, \cite{AleProc}.

Then, from \eqref{id:born_aprox_def} we obtain that
\begin{equation} \label{id:moment_prob_abstract}
\sigma_k [\ve]  = \lambda_k[V]-k, \qquad \forall k\in\N_0.
\end{equation}
Therefore, if $\ve$ exists, it must be a function/distribution in $\IB^d$ whose moments $\sigma_k[\ve]$ are the eigenvalues of $\Lambda_V-\Lambda_0$, \textit{i.e.} a solution to the Hausdorff moment problem \eqref{id:moment_prob_abstract}. 
This formal statement has also been obtained in \cite{BCMM22,BCMM23_n}, and implicitly  in \cite{DKN21_stability_steklov}.  
 
Given \eqref{id:moment_prob_abstract}, the existence of $\ve$ is, \textit{a priori}, far from being guaranteed: most  sequences of complex numbers are not sequences of moments  of any function; see \Cref{remark:characterization}. The fact that this is the case for a large class of radial potentials implies a strong necessary condition for an operator to be a DtN map; see \Cref{thm:characterization} for a precise statement. 

To prove that \eqref{id:moment_prob_abstract} can be solved for a large class of radial potentials, define the norm 
\begin{equation}   \label{id:q_basic_assumption}
\norm{F}_{\cV_d} : =    \sup_{j \in \N_0}  \int_{2^{-(j+1)}<|x|<2^{-j}} |F(x)||x|^{2-d} \, dx,
\end{equation}
and the associated Banach space of radial and real-valued functions:
\begin{equation}\label{id:def_Vd}
 \cV_d: = \{  V \in L^1_{\loc}(\IB^d\setminus \{0\};\R) : V  = q(|\cdot|),\; \norm{V}_{\cV_d }<\infty\}.
\end{equation}
This space is strictly larger than  $L^{d/2}_\rad(\IB^d;\R)$, since it includes the critical potential  $V(x) = c|x|^{-2}$ with $c\in \R$. In fact, $L^{d/2,\infty}_\rad(\IB^d;\R) \subset \cV_d$ for $d>2$ (see \Cref{sec:appendix_lorenz}).

Note that the DtN map is not always well-defined for every potential $V\in\cV_d$. It could happen that $0$ is in the Dirichlet spectrum of $-\Delta+V$ or that $-\Delta+V$ is not essentially self-adjoint. However, in the radial setting, one can give a meaningful definition of $\lambda_k[V]$ in terms of separation of variables that coincides with the spectral definition when $\Lambda_V$ exists (see \Cref{def:genspec}). In particular, the values $\lambda_k[V]$  can always be determined from a section of the Cauchy data of $-\Delta+V$ for all $k>k_V$, where 
\begin{equation} \label{id:beta_constant}
k_V :=  \beta_V -(d-2)/2, \qquad \beta_V := \frac{2}{|\IS^{d-1}|}\max\left( \sqrt{6|\IS^{d-1}|\norm{V}_{\cV_d} },3e \ \norm{V}_{\cV_d}   \right),
\end{equation}
(see \Cref{remark:CD}). The eventual ambiguity in the definition of $\lambda_k[V
 ]$ for some indices $k\leq k_V$  will not affect $\ve$ on $\IB^d\setminus\{0\}$. 
 
 Define the $L^1$-weighted space $L^1(\IB^d,|x|^{2s}dx)$ by the norm
 \begin{equation*}
     \norm{f}_{L^1(\IB^d,|x|^{2s} dx)} := \int_{\IB^d} |f(x)| |x|^{2s}  \,  dx,\qquad s>0,
 \end{equation*}
 and let $\cE'(\IB^d)$ be the set of  distributions on $\IB^d$ with compact support on $\IB^d$.

\begin{main-theorem}[Existence] \label{mt:existence}
Let $d\ge 2$ and $V \in \cV_d$; then the following holds. 
\begin{enumerate}[i)]
\item There exists a $k_0 \le k_V +1$ and a radial function $\ve\in L^1(\IB^d,|x|^{2k_0}dx)$ such that
    \begin{equation} \label{id:thm_1_1}
    \sigma_k[\ve] = \lambda_k[V]-   k,     \qquad \text{for all } \, k \geq k_0. 
    \end{equation}
    \item There exists a radially symmetric distribution $V^\mB_r \in \cE'(\IB^d)+L_\rad^1(\IB^d)$  such that
    \begin{equation} \label{id:thm_1_2}
            \sigma_k[V^\mB_r] = \lambda_k[V]-k,  \quad \text{ for all }  k\in\N_0.
    \end{equation}
\end{enumerate} 
In addition, $\ve$ and $\ve_r$ are, respectively, uniquely determined by \eqref{id:thm_1_1} and \eqref{id:thm_1_2}  and $\ve_r  = \ve$  in $\IB^d \setminus \{0\}$ in the   sense of distributions.
\end{main-theorem}

This theorem shows that the Born approximation $\ve$ is a well-defined function that exists for all potentials $V \in \cV_d$. As far as we know, these are the first rigorous results on the existence of solutions of \eqref{id:born_aprox_def} and \eqref{id:moment_prob_abstract} in the literature. 

 The solution $\ve$ is unique in a strong sense:  $\ve$ is always  uniquely determined by the values $\lambda_k[V]-k$ with $k>k_L$ for any given $k_L>0$. However, in some cases, $\ve$ can be  a singular function in $x=0$, which explains why \eqref{id:thm_1_1} only holds  in general for $k>k_V$ (see Section \ref{sec:explicitexamples} for explicit examples of such behavior). 
This motivates the introduction of the distribution $\ve_r$, which is a \textit{regularization} of the Born approximation in the sense of \cite[p. 11]{GelShilVol1}, since it is a distribution that coincides identically with $\ve$ when $x\neq 0$, and vanishes outside the ball. 
Either way, \eqref{id:thm_1_1} and \eqref{id:thm_1_2} provide two closely related  rigorous interpretations  of the formal identity \eqref{id:moment_prob_abstract}.
To state \eqref{id:thm_1_2} it is necessary to extend definition \eqref{id:moments_ball} to the distributions $\cE'(\IB^d)$ by duality, see \Cref{sec:regularization} for details.
 
One advantage of introducing the regularized Born approximation $\ve_r$ is that in \Cref{mt:regularization} we show that  there is an explicit expression to compute $\ve_r$ from the spectrum of the DtN map:
\begin{equation} \label{id:VB_regularization_intro}
    \widehat{\ve_r}(\xi)
    =  2 \pi^{d/2}  \sum_{k=0}^\infty  \frac{ (-1)^k}{k! \Gamma(k+d/2)}
    \left(\frac{|\xi|}{2}\right)^{2k}   (\lambda_k[V] - k),
\end{equation}
where the following convention for the Fourier transform of functions in $L^1(\IB^d)$ is used
\begin{equation*}
 \widehat{f}(\xi) = \cF (f) (\xi) := \int_{\IB^d} f(x) e^{-ix\cdot\xi} \, dx,
\end{equation*}
with its natural extension to $\cE'(\IB^d)$.
Identity \eqref{id:VB_regularization_intro} originates from a solution formula for the moment problem for compactly supported distributions (see \Cref{lemma:fourier_transform_identity} and \cite[Section~3]{BCMM22}) and it serves to reconstruct $\ve_r$ explicitly, and hence $\ve$ by restriction to $\IB^d\setminus\{0\}$, from the eigenvalues of $\Lambda_V$.   Formula \eqref{id:VB_regularization_intro} appeared originally in \cite[Theorem 1]{BCMM22} as a formal expression obtained  by linearizing a well-known formula to reconstruct $V$ from the DtN map using CGOs, and it has been used to numerically reconstruct $\ve$ in  \cite{BCMM23_n}.

\begin{remark} \label{remark:characterization}
The existence of solutions to the Hausdorff moment problem  is a subtle issue, since a sequence of moments must satisfy certain non-trivial necessary conditions (see, for example, \cite{Haus23} and \cite{Bor78}). Therefore,
the existence of the Born approximation obtained by \Cref{mt:existence} can be considered a kind of partial characterization of DtN operators. 
This is discussed in more detail in \Cref{sec:appendix_moments}. 
\end{remark}


\subsection{ Stable reconstruction of a potential from its Born approximation} \label{sec:reconstr_intro} \

We start by establishing that the correspondence  $\Phi_\mB(V)=\ve$ depicted in \eqref{eq:diagram_1} is injective. In fact, the Born approximation contains all the necessary information to reconstruct $V$ \textit{locally} in annuli from the boundary:
\begin{equation*}
    U_b:=\{x\in\IB^d\,:\,b<|x|<1\}.
\end{equation*}
\begin{main-theorem}[Uniqueness] \label{mt:uniqueness}
Let $d\ge 2$ and $V_1,V_2 \in \cV_d$. Denote $\ve_j$ the Born approximation of $V_j$, $j=1,2$. Then, for every $b\in(0,1)$,
\begin{equation*}
V_1^\mB|_{U_b} = V_2^\mB|_{U_b}\;   
\iff 
V_1|_{U_b} = V_2|_{U_b}.
\end{equation*}
\end{main-theorem}
Since $\ve$ is uniquely determined by $\Lambda_V$, this theorem immediately implies that the DtN map uniquely determines $V$. A simple consequence of \Cref{mt:uniqueness} is that  $V =0$ on $U_b$ if and only if $\ve =0$  on  $U_b$. See \cite[Theorem 1.4]{DKN21_stability_steklov} for a related local uniqueness result.
As mentioned previously, uniqueness in the Calderón problem was proved by Kohn and Vogelius \cite{KoVo84} in the analytic class and by Sylvester and Uhlmann \cite{SU87}  in the smooth class with $d \geq 3$.  For results in $d \geq 3$ with less regularity, see, for instance, \cite{Na88, Ch90, Na92,Brown96,HabTa,Ha15,CaroRogers16}.  
The planar Calderón problem for the conductivity equation was solved by Nachman~\cite{Na96} and by Astala and Päivärinta~\cite{AP06} for $\cC^2$ and $L^\infty$ conductivities, respectively.  For Schrödinger potentials, it was solved by Bukhgeim~\cite{Bukhgeim_2008} for $\cC^1$ potentials,  and  in \cite{Blasten_2015} for $L^p$ potentials.

We now turn to the question of stability.  The following theorem shows  that the nonlinear map $\Phi_\mB^{-1}: \ve \longmapsto V $  is Hölder-continuous in the $L^1(U_b)$ norm, provided we restrict $\Phi_\mB^{-1}$ to the subset of Born approximations $\ve$ such that $ \norm{V}_{L^p(U_b)} \le N$ with $p>1$.

\begin{main-theorem}[Local stability] \label{mt:stability_local}
    Let $d\ge 2$, $1<p\le \infty $, and $b\in (0,1)$. For all $N>|\IS^{d-1}|^{1/p}$, there exist  constants $C_{d,N,p,b}>0$ and $0<\delta_{d,p,b}<1$ such that, for all  $V_1,V_2 \in \cV_d$ satisfying
    \begin{equation} \label{e:pot_bound_local}
        \max_{j=1,2} \norm{V_j}_{L^p(U_b)} \le N, \qquad \qquad     \norm{\ve_1-\ve_2}_{L^1(U_b)} <\delta_{d,p,b} ,
    \end{equation}
    one has $\norm{V_1-V_2}_{L^1(U_b)} < C_{d,N,p,b} \, \norm{\ve_1-\ve_2}_{L^1(U_b)}^{\frac{p}{2p-1}}$.
\end{main-theorem}

The previous local estimate can be improved to a global estimate in the whole ball by using a weighted norm to deal with the eventual singularities of $\ve$ at the origin.
\begin{main-theorem}[Global stability] \label{mt:stability_global}
    Let $d\ge 2$, $d/2<p < \infty $, and $M > d-1$. 
    There exist $C_{d,M,p}>0$ and  $0<\alpha<1$, depending only on $M,d,p$, such that for every $V_1,V_2 \in L^{p}_\rad (\IB^d;\R)$ satisfying
    \begin{equation} \label{e:def_pot_bound}
        \begin{aligned}
            &\max_{j=1,2} \norm{V_j}_{L^p(\IB^d)} < \frac{|\IS^{d-1}|^{\frac{1}{p}}}{4} \frac{p-d/2}{p-1} M ,  \\
            &\int_{\IB^d}|\ve_1(x)-\ve_2(x)| |x|^{M-(d-2)} \,dx< |\IS^{d-1}|,
        \end{aligned}
    \end{equation}
    one has
    \begin{equation} \label{e:glob_sta}
       \norm{V_1-V_2}_{L^1(\IB^d)} <C_{d,M,p}\left(\int_{\IB^d}|\ve_1(x)-\ve_2(x)||x|^{M-(d-2)}\, dx\right)^\alpha.
    \end{equation} 
\end{main-theorem}
The fact that $\ve$ may be singular at $x=0$ means that the global stability estimate \eqref{e:glob_sta} only holds in a weighted $L^1$-space. As soon as the first estimate of \eqref{e:def_pot_bound} holds, one has $ \int_{\IB^d}|\ve_j(x)||x|^{M-(d-2)}\, dx<c\norm{V_j}_{L^p(\IB^d)}$ for $j=1,2$ and some $c>0$, so the estimate \eqref{e:glob_sta} is non-trivial (see \Cref{lemma:finitness_V}).  Also, notice that when one goes from the local to the global estimate, one loses the independence of the Hölder exponent from $M$.  
    

Let $G_{p,M}: = \{\ve: \norm{V}_{L^p(\IB^d)}< c_{d,p}M \}$, $p>d/2$, where $c_{d,p}$  is the factor in the RHS of the first inequality in \eqref{e:def_pot_bound}. \Cref{mt:stability_global} proves that  $\Phi_\mB^{-1} : \ve\longmapsto V$ is Hölder continuous from $G_{p,M} \subset L^1(\IB^d,|x|^{M-(d-2)}dx)$ to $L^1(\IB^d)$.
An important feature of this result is that it requires a very mild condition on the potentials; only a uniform $L^p(\IB^d)$ bound (the first inequality in \eqref{e:def_pot_bound}) is needed, which does not imply that the potentials lie in a compact subset of $L^1(\IB^d)$. In contrast, the modulus of continuity of $\Phi^{-1}$ is at best logarithmic, see \cite{Mand00}\footnote{The potentials given in \cite{Mand00} are not necessarily radial, but Mandache claims that even radial potentials give counterexamples to stability (see the remark before \cite[Lemma~4]{Mand00}).}, even under the usual regularity assumptions that guarantee the compactness of the set of  potentials.
We conclude from this that the linear operator $ d\Phi_0^{-1}:\Lambda_V-\Lambda_0 \longmapsto \ve$ is decompressing the information on $V$ contained in the DtN map and transforming the ill-posed inverse problem $\Phi^{-1}: \Lambda_V-\Lambda_0 \longmapsto V $ into the Hölder stable inverse problem  $\Phi_\mB^{-1} : \ve\longmapsto V$.

The proofs of Theorems \ref{mt:uniqueness}, \ref{mt:stability_local} and \ref{mt:stability_global} do not involve the construction of exponentially growing solutions (CGOs) commonly used in the Calderón problem. They arise from the approach to 1-$d$ inverse spectral theory, based on the notion of $A$-amplitude, introduced by Simon in \cite{IST1}, and later improved by Simon, Gesztesy, and Ramm in \cite{IST2,IST3} and by Avdonin, Mikhaylov, and Rybkin in \cite{AMR07,AM10}. 
It turns out that this object can also be interpreted as the Born approximation in the inverse problem of recovering a potential from the Weyl-Titchmarsh function of a Schrödinger operator (see \cite{MaMe24}). 
The approach of \cite{IST1} has also been applied in the context of the Steklov problem for  warped product manifolds in \cite{DHN21, DKN21_stability_steklov, DKN23, Gen2020, Gen2022}. In particular, the results in \cite{DKN21_stability_steklov} imply stability and uniqueness results for the radial Calderón problem, both for the conductivity and Schrödinger cases. 
Spectral theory methods have also  been used, for instance, in \cite{KoVo85,Syl_1992,MR3614929} for the radial Calderón problem and related inverse problems.

The ideas from \cite{IST1} can also be used to elaborate an explicit reconstruction method for the radial Calderón problem. We next present a simplified version of this method for radial $\cC^1(\IB^d)$ potentials (see \Cref{remark:reconstruction} for the general version).

\begin{main-algorithm}[Reconstruction]\label{alg:rec} Given $(\lambda_k[V])_{k\in \N_0}$ one reconstructs $V$ as follows.
\begin{enumerate}[1)]
    \item  Using \eqref{id:VB_regularization_intro} and $\ve = \ve_r$ in $\IB^d\setminus\{0\}$, one can  reconstruct $\ve$ from the eigenvalues of the DtN map $(\lambda_k[V])_{k\in \N_0}$ (or the Cauchy data when the DtN map is not well-defined). 
\item Find the unique $\cC^1$ solution of
\begin{equation}  \label{e:nonl}
     r\frac{\partial W}{\partial r} (r,s) -s\frac{\partial W}{\partial s} (r,s)   =  s^2\int_{r}^1 W\left (\frac{r}{\nu},s\right)W(\nu,s) \frac{d\nu}{\nu}, \qquad r,s \in (0,1),
\end{equation}
such that $W(|x|,1) = \ve(x)$. As it turns out, $W(|x|,s) = s^{-2}[V_s]^\mB(x)$ where $V_s(x) = s^2V(sx)$. In other words, $W(\cdot,s)$ is the radial profile of the Born approximation of the dilated potential $V_s$.
\item Once $W(r,s)$ is known,  use that   
   $  W(1,|x|) =  V(x)$. This holds since, by \Cref{mt:approximation} below, the Born approximation always coincides with the potential at the boundary of $\IB^d$.
\end{enumerate}
\end{main-algorithm}
In this algorithm, one reconstructs the potential layer  by layer:  the information on $V$ that is already known is taken outside the ball by dilating the potential. This is reminiscent of the so-called  layer stripping methods used in EIT, see \cite{CNS20_layer_stripping} for references. 
One can certainly replace step 1) with any other suitable method to solve the moment problem.
One might also replace step 2) by adapting the approach in \cite{AM10}, based on the boundary control method for the wave equation, which has the advantage of relying on a linear integral equation.


\subsection{Structure and approximation properties of the Born approximation}\

We now turn to investigate the qualitative behavior of  $\ve$ and its connections to the potential $V$.
  \begin{main-theorem}[Approximation properties] \label{mt:approximation}
Let $V \in \cV_d$, $d\ge 2$, such that  $V = q(|\cdot|)$, and let
\begin{equation*}
    \alpha(r) := \min \left( \beta_V, \,\int_r^1 s|q(s)|   \, ds \right).
\end{equation*}
Then  $\ve = V + F(|\cdot|)$, where $F$ is a continuous function in $(0,1]$ that  satisfies:
\begin{equation} \label{est:v_VB_estimate}
\left|F(r) \right| \le \frac{1}{r^{\alpha(r)+2}} \left( \int_r^1 s|q(s)|   \, ds \right)^2,\qquad F(1) = 0.
\end{equation}
In addition, if $q$ is $\cC^{m}$ in $(b,1]$ with $m\in \N_0$, then $F$ is in $\cC^{m+2}$ in $(b,1]$ and $F'(1^-) =0$.  
 \end{main-theorem}
 This theorem shows that $\ve$ approximates $V$ in $U_b$ with an error that depends only on $b$ and the size of the potential in $U_b$. It also implies that $\ve$ contains the discontinuities of the potential, a fact that can be observed in the numerical reconstructions of $\ve$ in \cite{BCMM23_n}.  This kind of recovery of singularities property of the Born approximation  is well known in scattering problems, see for example \cite{Paivarinta_Somersalo_1991, Ru01,Me19}. For other results of recovery of singularities  in the context of the Calderón problem, see \cite{GLSSU2018}. 

\subsection{Structure of the paper}
In \Cref{sec:Dtn_radial} we introduce the necessary background on the radial Calderón  Problem.    \Cref{mt:existence}(i), \Cref{mt:uniqueness}, and \Cref{mt:approximation} are based on the works \cite{IST1,AMR07}. All these results are proved in \Cref{sec:spectral_theory}. 
In \Cref{sec:reconstruction} we discuss a reconstruction method for the radial Calderón problem based on the Born approximation; in particular, the validity of \Cref{alg:rec} is proved there. We then use this to address in \Cref{sec:stability} the stability results for the Born approximation, Theorems \ref{mt:stability_local} and \ref{mt:stability_global}, which are based on a new stability result for the $A$-amplitude  of inverse spectral theory given in \Cref{thm:stability_A_function}. 
The regularization of the  Born approximation is introduced in \Cref{sec:regularization}, together with the proof of \Cref{mt:existence}(ii). Finally, \Cref{sec:appendix_moments} presents the consequences of our result towards giving a characterization of the set of radial DtN maps.

\begin{acknowledgements}
    FM and CJM acknowledge the support of Agencia Estatal de Investigación (Spain) through grant PID2021-124195NB-C31. FN is supported by the French Gdr Dynqua.
\end{acknowledgements}


\section{The direct problem in the radial case} \label{sec:Dtn_radial}

In this section, we present several useful facts on the Dirichlet problem for Schrödinger operators with potentials $V\in\cV_d$, as well as the precise definition of the sequence $(\lambda_k[V])_{k\in\N_0}$ when the DtN map is not well-defined. 

\subsection{The DtN map, Cauchy data, and the Fréchet derivative} \ \label{sec:subsec_technical_calderon}
Consider the Dirichlet problem
\begin{equation}  \label{id:calderon_direct_ball}
\left\{
\begin{array}{rlr}
-\Delta u +Vu =& 0   &\text{ on }   \IB^d,\\
u |_{  \IS^{d-1}} =& f, &    \\
\end{array}\right.
\end{equation}
where $f\in H^{1/2}(\IS^{d-1})$ and $\IS^{d-1} = \partial \IB^d$.

The general version of the Calderón Problem for the Schrödinger equation consists of determining the potential  $V$ in \eqref{id:calderon_direct_ball}  from the Cauchy data  
\begin{multline} \label{id:cauchy_data}
    \cC(V) =\{ (f,\partial_\nu u|_{\IS^{d-1}} ) \in H^{1/2}(\IS^{d-1})\times H^{-1/2}(\IS^{d-1}): \\
    u \in H^1(\IB^d) \text{ is a solution of \eqref{id:calderon_direct_ball}}\}.
\end{multline}
Here $\partial_\nu u|_{\IS^{d-1}} $ is defined in the weak sense using that
 for all $f,g\in H^{1/2}(\IS^{d-1})$ 
 \begin{multline}   \label{id:weak_def_dtn}
 \br{g, \partial_{\nu} u|_{\IS^{d-1}}}_{H^{1/2}(\IS^{d-1}) \times H^{-1/2}(\IS^{d-1})}  \\ =
 \int_{\IB^d} \nabla u (x)  \cdot \nabla v(x) \, dx  + \int_{\IB^d} V(x)  u (x)  v(x) \, dx,
 \end{multline}
where $u$ solves \eqref{id:calderon_direct_ball}  and $v \in H^1(\IB^d)$ is any function such that $v|_{\IS^{d-1}} =g$.

The weak formulation of \eqref{id:calderon_direct_ball} and the expression \eqref{id:weak_def_dtn} are well defined if one assumes 
\begin{equation} \label{id:well_posedness}
    V \in L^{p}(\IB^d;\R) \, \text{ with } p>1 \text{ and } p\ge d/2.
\end{equation} 
In addition, if one also requires that
\begin{equation} \label{id:basic_assumption}
0\notin \Spec_{H^1_0(\IB^d)} \left( -\Delta + V \right),
\end{equation}
then there is a unique solution $u\in H^1(\IB^d)$ of \eqref{id:calderon_direct_ball} for each $f\in H^{1/2}(\IS^{d-1})$. When \eqref{id:basic_assumption} holds, one can define the Dirichlet to Neumann (or DtN map), a bounded operator
\begin{equation*}
    \Lambda_V : H^{1/2}(\IS^{d-1}) \to  H^{-1/2}(\IS^{d-1}),
\end{equation*} 
by $\Lambda_V (f) := \partial_{\nu} u|_{\IS^{d-1}}$.
In this case, the Cauchy data $\cC(V)$ coincides with the graph of the DtN map. 

Let $u_j\in H^1(\IB^d)$, $j=1,2$, be the solutions of \eqref{id:calderon_direct_ball} with $V = V_j$ and $f = f_j$, and assume \eqref{id:basic_assumption} holds for both potentials. Then,  from \eqref{id:weak_def_dtn}, one can obtain the identity
 \begin{equation} \label{id:aless}
 \br{f_1, (\Lambda_{V_1}-\Lambda_{V_2}) f_2}_{H^{1/2}(\IS^{d-1}) \times H^{-1/2}(\IS^{d-1})} =  \int_{\IB^d} (V_1(x)-V_2(x))  u_1 (x)  u_2(x) \, dx,
 \end{equation}
 known as Alessandrini's identity. Taking $V_2 =0$ and differentiating, it follows that the Fréchet derivative $d\Phi_0$ satisfies
 \begin{equation*}
 \br{f_1, d\Phi_0(V) f_2}_{H^{1/2}(\IS^{d-1}) \times H^{-1/2}(\IS^{d-1})} =  \int_{\IB^d} V(x)  w_1 (x)  w_2(x) \, dx,
 \end{equation*}
 where $w_j$ satisfies $-\Delta w_j =0 $ on $\IB^d$ and $w_j = f_j$ on  the boundary.  The previous identity is well defined for $V$ satisfying \eqref{id:well_posedness} and, if $V$ is radial, choosing $\ol{f_1} = f_2 = Y_k$ for a spherical harmonic $Y_k \in \gH_k$ one can prove \eqref{e:frechet_dif}. Under the radial assumption,  $d\Phi_0(V)$ is a compact operator in $\cL(L^2(\IS^{d-1}))$ assuming only $V\in L^1(\IB^d)$, since $\sigma_k[V] \to 0$ as $k\to \infty$.

\subsection{Generalized DtN eigenvalues in the radial case}  \ \label{sec:subsec_DtNmap} 
Note that $\cV_d$ contains potentials, for which the weak formulation \eqref{id:weak_def_dtn} does not make sense (and thus even the Cauchy data \eqref{id:cauchy_data} are not well-defined); the operator $-\Delta+V$ could even fail to be essentially self-adjoint (for instance, when $V(x)=c|x|^{-2}$ with $0\leq c + \frac{(d-1)(d-3)}{4}< \frac{3}{4}$, see (\cite{RS2}, Theorem X.11)). The rest of this section is devoted to showing that even in these situations one has a well-defined sequence $(\lambda_k[V])_{k\in\N_0}$  from which the Born approximation can be defined.

Assume for the moment that $V(x) = q(|x|)$ for some measurable function $q:(0,1) \to \R$ such that \eqref{id:well_posedness} holds.
Recall that, as in the introduction, $\gH_k$ stands for the subspace of spherical harmonics of degree $k$  in $\IS^{d-1}$. 
Let   $Y_k \in\gH_k$ then:
\begin{equation*}
-\Delta_{\IS^{d-1}} Y_k(\omega) = k(k+d-2) Y_k(\omega), \qquad \omega \in \IS^{d-1}.
\end{equation*}

Taking $f = Y_k$ in \eqref{id:calderon_direct_ball} and using a Fourier expansion in spherical harmonics of $u$, it yields that $u(x) = b_k(|x|) Y_k(x/|x|)$, where $b_k$ is a solution of
\begin{equation} \label{id:radial_b}
    -\frac{1}{r^{d-1}}\frac{d}{dr}\left(r^{d-1} \frac{d}{dr} b_k(r)\right)+ \left(\frac{k(k+d-2)}{r^2}+q(r)\right)b_k(r)=0,
\end{equation}
subject to the boundary condition $b_k(1) = 1$. 

Note that if, in addition to \eqref{id:well_posedness}, the potential satisfies \eqref{id:basic_assumption} then there exists a unique solution $b_k$ such that $b_k(|x|) Y_k(x/|x|)$ belongs to $H^1(\IB^d)$. Since the normal derivative coincides with the radial derivative in spherical coordinates, it turns out that the DtN map is well-defined and satisfies that
\begin{equation*}
    \Lambda_V (Y_k) = \partial_r b_k (1) Y_k.
\end{equation*}
This shows that $Y_k$ is an eigenfunction of $\Lambda_V$, and  that   the space $\gH_k$ is an invariant subspace of the DtN map operator  with eigenvalue 
\begin{equation}  \label{id:dtn_eigenvaule}
 \lambda_k[V] : = \partial_r b_k (1).   
\end{equation} 

We now want to understand how this can be generalized when conditions \eqref{id:well_posedness} or \eqref{id:basic_assumption} fail and the DtN map is not well-defined. The key technical point to achieve this is contained in the following result.  We recall that $k_V$ is the constant defined in \eqref{id:beta_constant}.
\begin{lemma} \label{lemma:radial_solutions}
Let $d\ge 2$, $V\in\cV_d$. Then, for every $k> k_V$  there is a unique solution $b_k$  of \eqref{id:radial_b} with $b_k(1) = 1$ such that for every $Y_k \in \gH_k$, the function $u_k(x) = b_k(|x|) Y_k(x/|x|)$ satisfies that
\begin{equation} \label{id:u_k_condition}
u_k  \in L^2(\IB^d, |x|^{-2}dx).
\end{equation}
\end{lemma}
\begin{remark}
When the conclusion of \Cref{lemma:radial_solutions} holds, then  $u_k$ is an $H^1_{\loc}$ solution to
\begin{equation}  \label{id:calderon_direct_Yk}
\left\{
\begin{array}{rlr}
-\Delta u_k +Vu_k =& 0   &\text{ on }  \, \IB^d\setminus \{0\},\\
u_k |_{  \IS^{d-1}} =& Y_k. &    \\
\end{array}\right.
\end{equation}
In fact $u_k \in \cC^1(\IB^d\setminus 0)$, since $b_k\in\cC^1(0,1)$ and the spherical harmonic is smooth.
\end{remark}

\begin{remark}\label{remark:CD}
  If, in addition, $V \in L^p(\IB^d)$ with $p>1$ and $p\ge d/2$, then the solutions  $u_k$ obtained in \Cref{lemma:radial_solutions} are proper weak  $ H^1(\IB^d)$ solutions of \eqref{id:calderon_direct_Yk}; and given any  other solution $u \in H^1(\IB^d)$ of \eqref{id:calderon_direct_Yk} with the same $k$  it holds that
\begin{equation} \label{id:lambda_determination}
    \partial_r b_k (1) = (Y_k, \partial_\nu u|_{\IS^{d-1}} )_{L^2(\IS^{d-1})},
\end{equation}
for all $Y_k \in \gH_k$ with $\norm{Y_k}_{L^2(\IS^{d-1})} =1$. This assertion is proved using standard arguments from the theory of linear elliptic equations, see \Cref{lemma:radial_solutions_appendix} in \Cref{sec:appendix_lemma}.
\end{remark}
The proof of \Cref{lemma:radial_solutions} is given in \Cref{sec:subsec_potentials}. This result shows, in particular, that for a potential satisfying \eqref{id:q_basic_assumption}, problem \eqref{id:calderon_direct_Yk} always possesses a unique solution of separation of variables such that \eqref{id:u_k_condition} holds, at least for $k$ large enough (even when \eqref{id:calderon_direct_Yk} is not well-posed).
This motivates the following definition.
\begin{definition}\label{def:genspec}
Let $V\in\cV_d$ and denote by $\cB_V$ the set of indices $k\in\N_0$ such that the conclusion of $\Cref{lemma:radial_solutions}$ fails. We define: 
\begin{equation} \label{id:eigenvalues_pot}
 \lambda_k[V] := \left\{ \begin{array}{ll}
   \partial_r b_k (1),   &  k\in\N_0\setminus\cB_V\\
    k,  & k\in \cB_V.
 \end{array}\right.
\end{equation}
\end{definition}
Note that \Cref{lemma:radial_solutions} states that $\cB_V$ is at most finite. In addition, if \eqref{id:well_posedness} holds, solutions $u_k$ with $k>k_V$ constitute a well-defined  section of the Cauchy data \eqref{id:cauchy_data},  and for each of those $k$, \Cref{remark:CD} ensures that $\lambda_k[V] $ can be determined from \textit{any} solution of \eqref{id:calderon_direct_Yk} via \eqref{id:lambda_determination}. If in addition \eqref{id:basic_assumption} holds, then $(\lambda_k[V] )_{k\in\N_0}$ coincides  exactly with the spectrum of the DtN map.


\subsection{Reduction to a Schrödinger operator on the half-line}  \ \label{sec:subsec_potentials}
In this section we prove \Cref{lemma:radial_solutions}.
Let $r=e^{-t}$; then, writing  
\begin{equation} \label{id:q_to_Q}
Q(t) := e^{-2t}q(e^{-t}), 
\end{equation}
so that $q(r) =r^{-2}Q(-\log r)$, we have the following. A function $b_k$ is a solution to \eqref{id:radial_b} if and only if
\begin{equation} \label{id:change_variables_radial}
v_k(t) := e^{-\frac{d-2}{2}t} b_k(e^{-t}),
\end{equation}
solves the following boundary value problem on the half-line:
\begin{equation}   \label{id:schrodinger_prob}
\left\{
\begin{array}{rlr}
-v_k''  + Q    v_k =& -\left(k+ \frac{d-2}{2}\right)^2 v_k  &\text{ on }   \R_+,\\
v_k(0) =& 1. &    \\
\end{array}\right.
\end{equation} 

We will show that for potentials $Q$ satisfying that
\begin{equation} \label{id:assumption_A2}
\normm{Q} := \sup_{y>0} \int_y^{y+1} |Q(t)| \, dt <\infty,
\end{equation}
problem \eqref{id:schrodinger_prob} possesses a unique solution provided that $k$ is big enough (see \Cref{lemma:Schrodinger_exist}). Performing a change of variables will yield \Cref{lemma:radial_solutions}.

Assumption \eqref{id:assumption_A2} comes from \cite[Theorem 2]{AMR07} and essentially corresponds to case 2  in \cite{IST3}. In particular it implies that the operator 
\begin{equation} \label{id:operator}
 -\frac{d^2}{dt^2} + Q,
\end{equation}
(with Dirichlet boundary condition at $t=0$), is in the limit point case at infinity. This operator is essentially self-adjoint on $\cC^\infty_c(\R_+)$ and bounded from below, (see e.g \cite{Eastham72}, \cite{RS2}, Theorem X.7 ). This condition motivates the introduction of the norm $\norm{\cdot}_{\cV_d}$ in \eqref{id:q_basic_assumption}.

\begin{remark}\label{remark:norms}
Let $V(x) = q(|x|)$ on $\IB^d$ and $Q(t) = e^{-2t}q(e^{-t})$.
Then
\begin{equation} \label{id:norm_equivalence}
\frac{1}{3}|\IS^{d-1}|\normm{Q} \le \norm{V}_{\cV_d} \le |\IS^{d-1}|\normm{Q}.
\end{equation}
In fact it follows from \eqref{id:q_basic_assumption} that
\begin{equation*}
   \norm{V}_{\cV_d}
 = |\IS^{d-1}| \sup_{j\in \N_0}\int_{j\log 2}^{ (j+1)\log2}|Q(t)| \, dt,  \qquad j\in \N_0.
\end{equation*}
\end{remark}
The space $\cV_d$ contains the radial functions in Lorentz  space $ L^{d/2, \infty}(\IB^d)$ with $d>2$, (the weak $L^{d/2}(\IB^d)$ space), see \Cref{sec:appendix_lorenz}.

Define the constant
\begin{equation} \label{id:beta_constant_Q}
\beta_Q := 2\max\left( \sqrt{2 \normm{Q} },e\normm{Q}   \right) .
\end{equation}
\begin{lemma} \label{lemma:Schrodinger_exist} 
  Let $Q \in L^1_{\loc}(\R_+)$ such that $\normm{Q} <\infty$, and consider the equation 
  \[-u_z''  + Q    u_z = z u_z.\]
  Then, for all $z\in \IC\setminus{[-{\beta_Q}^2,\infty)}$ 
  there exists a unique solution $u_z$ such that $u_z(0) = 1$ and $u_z \in L^2(\R_+)$.
\end{lemma}

\begin{proof}
Under condition (\ref{id:assumption_A2}), it is known that $-\frac{d^2}{dx^2} +Q$ is limit point at infinity, (see \cite{IST1}). Thus, for $\Im z \not=0$, there exists  a unique solution $u_z$ with $u_z (0)=1$  which is $L^2$ at infinity. Moreover, for $\Im z \not=0$ one has, (see e.g \cite{Teschl}, Lemma 9.14):
\begin{equation} \label{id:quotient}
\Im M(z) = \Im z \int_0^{+\infty} |u_z (x) |^2 \ dx,
\end{equation}
where $M(z)$ is the so-called Weyl-Titchmarsh function (see Section \ref{sec:spectral_theory} for details). One has $\overline{M(z)} = M(\bar{z})$ and under the assumption (\ref{id:assumption_A2}), the map $k \to M(-k^2)$ has an analytic continuation to $\Re k > \beta_Q$, (see \cite{AMR07}, Section 5, Algorithm 1). 

For a fixed  $k>\beta_Q$, ($k$ real), and for $\epsilon >0$ small enough, we set
\begin{equation}
    f(\epsilon) := \Im M(-k^2+i \epsilon).
\end{equation}
Clearly, $f$ is smooth, $f(0)=0$ and using (\ref{id:quotient}), one gets
\begin{equation*}
    \frac{f(\epsilon)}{\epsilon} = \int_0^{+\infty} |u_{-k^2+i\epsilon} (x) |^2 \ dx.
\end{equation*} 
Taking $\epsilon \to 0$ and using Fatou's lemma, we see that $u_{-k^2}$ is $L^2$ at infinity. Uniqueness follows from the fact that $-\frac{d^2}{dx^2} +Q$ is limit point at infinity.
\end{proof}

We can now prove \Cref{lemma:radial_solutions}.
 \begin{proof}[Proof of $\Cref{lemma:radial_solutions}$]
As we have seen, using the change of variables  \eqref{id:change_variables_radial} in \eqref{id:radial_b},   the function $v_k(t) = e^{-\frac{d-2}{2}t} b_k(e^{-t})$     satisfies  \eqref{id:schrodinger_prob} with $Q(t) = e^{-2t}q(e^{-t})$. Since $V = q(|\cdot|)$, by  \eqref{id:norm_equivalence} we know that $\norm{V}_{\cV_d} <\infty$ implies $\normm{Q} <\infty$. Therefore, by \Cref{lemma:Schrodinger_exist}, for all $k + (d-2)/2> \beta_Q$ there exists a unique solution $v_k$ of \eqref{id:schrodinger_prob} such that $v_k \in L^2(\R_+)$. Also, from  \eqref{id:beta_constant}, \eqref{id:beta_constant_Q} and \eqref{id:norm_equivalence}, it follows that $\beta_Q \le \beta_V$.

Now, using that $ b_k(r) = r^{-\frac{d-2}{2}}  v_k(-\log r) $ it follows that
\[
\norm{v_k}_{L^2(\R_+)}^2 =  \int_0^1 |b_k(r)|^2 r^{d-3} \, dr = \norm{Y_k}_{L^2(\IS^{d-1})}^{-2}  \int_{\IB^d} |u_k(x)|^2 \frac{1}{|x|^2} \, dx,
\]
where $u_k(x) = b_k(|x|) Y_k(x/|x|)$. Thus, $b_k$ is the unique solution of \eqref{id:radial_b}  such that  \eqref{id:u_k_condition}  holds. 
\end{proof}


\section{Connection with inverse spectral theory and Simon's A-amplitude}  \label{sec:spectral_theory}

This section is devoted to the proofs of Theorems \ref{mt:existence}(i), \ref{mt:uniqueness}, \ref{mt:approximation}.  This will be done by establishing a link between Simon's approach to inverse spectral theory for Schrödinger operators on the half-line and the radial Calderón problem. We also give a monotonicity result for the Born approximation in \Cref{thm:monotonicity}.

\subsection{The DtN map and Weyl-Titchmarsh function} \
By \Cref{lemma:Schrodinger_exist} 
if $Q$ satisfies \eqref{id:assumption_A2},  the Schrödinger equation
\begin{equation*}    
-u_z''  + Q    u_z = z u_z, \qquad \text{ on }   \R_+,
\end{equation*}
has a unique solution $u_z \in L^2(\R_+)$ up to a multiplicative constant whenever $\Im(z)>0$.

The Weyl-Titchmarsh function $M(z)$ associated with the half-line Schrödinger operator is defined as
\begin{equation*}
    M(z) := \frac{u_z'(0)}{u_z(0)},  \qquad z \in \IC_+ : = \{ \Im(z)>0 \}.
\end{equation*}
$M(z)$ is analytic in $\IC_+$,  and, under the assumption $\normm{Q}<\infty$, $M(z)$ has an analytic continuation to $\IC\setminus [-\beta_Q^2,\infty)$, where recall 
 $\beta_Q= 2\max( \sqrt{2\normm{Q}},e\normm{Q}  )$ (see Lemma \ref{lemma:Schrodinger_exist}).
Therefore, by  \eqref{id:schrodinger_prob} we have that
\begin{equation*}
    v_k'(0) = M\left ( -\kappa_k^2  \right )  ,
\end{equation*}
where, for simplicity, we introduce the notation
\begin{equation*}
\kappa_k = k+ \frac{d-2}{2}, \qquad k\in \N_0.
\end{equation*}

On the other hand, using \eqref{id:dtn_eigenvaule} and inverting the change of variables \eqref{id:change_variables_radial} one obtains
\begin{equation*}  
\lambda_k[V]  =  \partial_r b_k (1)   =  \partial_r \left[r^{-\frac{d-2}{2}}v_k(-\log r)\right ]\bigg |_{r=1}  = -\frac{d-2}{2} -v'_k(0).
\end{equation*} 
From this, it follows that
\begin{equation} \label{id:lambda_weyl}
\lambda_k[V] = -\frac{d-2}{2} -M(-\kappa_k^2), \qquad k>\beta_Q-\frac{d-2}{2}.
\end{equation}
This shows that when \eqref{id:q_to_Q} holds, the eigenvalues of the DtN map of $V$ coincide with the values of the $M$-function of $Q$ on a certain discrete set.


\subsection{Simon's A-amplitude}  \
Simon proved in \cite{IST1}, and was later refined in \cite{AMR07} assuming just that $Q$ satisfies \eqref{id:assumption_A2}, that there exists a function  $A\in L^1_{\loc}(\R_+)$ such that  
\begin{equation} \label{id:laplace_treansform_A}
M(-\kappa^2) = -\kappa -\int_{0}^\infty A(t) e^{-2\kappa t} \, dt \quad \text{ for } \Re (\kappa)>\beta_Q,
\end{equation}
where the integral is absolutely convergent. The function $A$ is called the $A$-amplitude of $Q$.
This function enjoys a series of interesting properties. 
\begin{theorem}[Theorem 1.5 \cite{IST1} and \cite{AMR07}] \label{thm:simon_uniqueness}
Under the assumption \eqref{id:assumption_A2}, $Q$ on $[0,a]$ is only a function of $A$ on $[0,a]$. More precisely 
\[
Q_1(t) = Q_2(t) \;   \text{ a.e. on } \, [0,a] \, \iff A_1(t) = A_2(t) \;   \text{ a.e. on } \, [0,a].
\]
\end{theorem}

\begin{theorem}[Simon \cite{IST1} and \cite{AMR07}] \label{thm:simon_avdonin_error}
Assume $Q$ satisfies \eqref{id:assumption_A2}. Then $A \in L^1_{\loc}(\R_+)$, and
\[
A(t) = Q(t) + E(t), 
\]
where $E \in \cC(\R_+)$ satisfies, for every $t>0$,
\begin{equation} \label{est:a_Q_estimate_1}
\left|E(t) \right| \le \frac{1}{2} \left( \int_0^t |Q(s)|   \, ds \right)^2\left\{ e^{2\sqrt{2 \normm{Q}} \, t} + \frac{1}{\sqrt{2\pi}}e^{2e\normm{Q} t} \right\},
\end{equation}
and
\begin{equation} \label{est:a_Q_estimate_2}
\left|E(t) \right| \le   \left( \int_0^t |Q(s)|   \, ds \right)^2 \exp \left( t\int_0^t |Q(s)|   \, ds \right).
\end{equation}
In addition, if $Q$ is of class $\cC^m$, $m\in \N_0$ in $(0,a)$, then $E$ is of class  $\cC^{m+2}$ in $(0,a)$.
\end{theorem}
The previous theorem implies that the difference between $E=A-Q$ is small  close to the origin and that eventually $E(0)=0$. It also provides a recovery of singularities result.  Estimate \eqref{est:a_Q_estimate_2} shows that the error only depends locally on $Q$, while \eqref{est:a_Q_estimate_1} provides a global control of the growth of the exponential factor when $\normm{Q} <\infty$.


\subsection{From the \texorpdfstring{$A$}{A}-amplitude to the Born approximation}
We are now ready to prove our claims on existence and structure of the Born approximation.
\begin{proof}[Proof of $\Cref{mt:existence}(i)$]
We can combine the relation between the eigenvalues of the DtN map and the $M$ function given by \eqref{id:lambda_weyl} with the representation of the $M$ as a Laplace transform given in \eqref{id:laplace_treansform_A} to obtain that
\begin{equation} \label{id:abs_convergent}
\lambda_k[V] = k + \int_0^\infty A(t) e^{-2\left(k+\frac{d-2}{2} \right)t} \,dt, \quad \text{ for all } \; k>k_Q ,
\end{equation}
where $k_Q:= \beta_Q -\frac{d-2}{2}$. Using   the change of variables $r=e^{-t}$, we have
\begin{equation*} 
\lambda_k[V] -k =  \int_0^1 A(-\log r) r^{2k + d-3} \,dr, \quad \text{ for all } \; k>k_Q  , 
\end{equation*}
which can also be written as
\begin{equation} \label{id:abs_convergent_2}
\lambda_k[V] -k =  \frac{1}{|\IS^{d-1}|} \int_{\IB^d} |x|^{2k} \frac{A(-\log |x|)}{|x|^{2}} \, dx \quad \text{ for all } \; k>k_Q  .
\end{equation}

Recall that, {\it formally}, $\ve$ should be a solution of the moment problem \eqref{id:moment_prob_abstract}. The previous expression implies that there exists such a solution  for all $k >k_Q  $, since we can take
\begin{equation}   \label{id:VB_def}
\ve(x) := \frac{A(-\log |x|)}{|x|^2}.
\end{equation}
Notice that this is a actual solution of the problem, since  the fact that \eqref{id:abs_convergent} converges absolutely implies that also \eqref{id:abs_convergent_2} is absolutely convergent. Thus, we finally have that
\begin{equation}   \label{id:existence} 
\lambda_k[V] -k =  \frac{1}{|\IS^{d-1}|} \int_{\IB^d} |x|^{2k} \ve(x) \, dx \quad \text{ for all } \; k>k_Q  .
\end{equation}
Uniqueness is proved in \Cref{sec:appendix_moments} (see in particular identity \eqref{id:moment_vb}), and the identity \eqref{id:thm_1_1} follows from \eqref{id:existence}, since  by \Cref{remark:norms} one always has $\beta_Q \le \beta_V$ and $k_Q \le k_V$ when \eqref{id:q_to_Q} holds.
\end{proof}

\begin{proof}[Proof of \Cref{mt:uniqueness}]
It is  a direct consequence of \Cref{thm:simon_uniqueness} using that 
\begin{equation} \label{id:V_VB_transform}
    V(x)  = |x|^{-2}Q(-\log|x|), \quad \text{and} \quad  \ve(x)  = |x|^{-2}A(-\log|x|). \qedhere
\end{equation}
\end{proof}

\begin{proof}[Proof of \Cref{mt:approximation}]
Using the change of variables $q(r) = r^{-2}Q(-\log r)$, it follows that
\begin{equation*}
    \int_0^t |Q(u)|   \, du = \int_r^1  s|q(s)| \, ds.
\end{equation*}
Therefore, since $F(r) = r^{-2}E(-\log r)$, estimate \eqref{est:a_Q_estimate_1} becomes
\begin{multline*}
    \left|F(r)\right|
    \le\frac12\left(\int_r^1 s|q(s)|\,ds\right)^2\left\{r^{-2-2\sqrt{2\,\normm{Q}}}+\frac1{\sqrt{2\pi}}\,r^{-2-2e\normm{Q}}\right\} \\\le 
\left(\int_r^1 s|q(s)|\,ds\right)^2r^{-2-\beta_V},
\end{multline*}
since $\beta_Q \le \beta_V$. Also,  \eqref{est:a_Q_estimate_2} becomes
\begin{equation*}  
\left|F(r) \right| \le   \left( \int_r^1 s|q(s)|   \, ds \right)^2 r^{-2- \int_r^1 s |q(s)|   \, ds },
\end{equation*}
which together yield \eqref{est:v_VB_estimate}. That $ F(1) = 0 $ follows directly from these estimates. The fact that, if  $q$ is $\cC^{m}$ in $(b,1]$ with $m\in \N_0$, then $F$ is in $\cC^{m+2}$ in $(b,1]$ it follows from the last statement of \Cref{thm:simon_avdonin_error}. Since $F$ is $\cC^{m+2}$ close to the boundary under these assumptions, the estimates imply that also  $F'(1^{-})=0$.
\end{proof}

We now prove a simple monotonicity property of the Born approximation.
\begin{proposition}[Monotonicity] \label{thm:monotonicity}
Let $d\ge 2$ and $V_1,V_2 \in \cV_d$. Then, 
\[
V_1(x) \le - V_2(x) \; \text{ on } \, U_b \Longrightarrow \; V_1^\mB (x) \le - V_2^\mB(x)  \; \text{ on } \, U_b ,
\]
for any $0<b <1$.
\end{proposition}
\begin{proof}
    It is a direct consequence of \cite[Theorem 10.2]{IST2} together with \eqref{id:V_VB_transform}.
\end{proof}


\subsection{Some explicit examples.} \  \label{sec:explicitexamples}
We present two examples for which the Born approximation $\ve$ can be computed explicitly. They show in particular that the Born approximation can effectively be more singular at the origin than the potential $V$.

First, let us consider the so-called Bargmann potentials in $\R_+$:
 \begin{equation*}
 Q(t) = -8\mu^2 \frac{\mu -\nu}{\mu +\nu} \frac{e^{-2\mu t}}{\left(1+ \frac{\mu-\nu}{\mu +\nu} e^{-2\mu t} \right)^2},
 \end{equation*}
 where $\mu >0$, and $\nu\ge 0$. Then, in \cite[Section 11]{IST2} it is shown that for $s \geq 0$,
 \begin{equation*}
A(s) =  2(\nu^2-\mu^2)e^{-2\nu s}.
 \end{equation*}

Therefore, using  \eqref{id:V_VB_transform} one gets:
 \begin{equation*}
 V(x) = -8\mu^2 \frac{\mu -\nu}{\mu +\nu} \frac{|x|^{2(\mu-1)}}{\left(1+ \frac{\mu-\nu}{\mu +\nu} |x|^{2\mu} \right)^2},
 \end{equation*}
 and
 \begin{equation*}
\ve (x) =  2(\nu^2-\mu^2)|x|^{2(\nu-1)}.
 \end{equation*}
When $\mu \geq 1$, $V$ is a continuous function on $\IB^d$ whereas the Born approximation $\ve$ has a singularity at the origin if $\nu <1$.

Secondly, let us consider the potential defined in the unit ball $\IB^d$ by
\[
V(x) = \frac{q_0}{|x|^2}, \quad q_0 \in \R.  
\]
As discussed previously (see also \Cref{sec:appendix_lorenz}), $V \in L^{\frac{d}{2}, \infty}(\IB^d)$, (and if $q_0$ is small enough, the DtN  map is well defined, see \Cref{remark:norms_2}). This potential corresponds, by the change of variables \eqref{id:q_to_Q}, to the case $Q(t)=q_0$, $t \in \R_+ $,  which was studied in \cite[Theorem 10.1]{IST2} to conclude the following. 

If $q_0 >0$, the Born approximation is given by 
\begin{equation} \label{id:Bornpotential2}
\ve(x) = - \frac{\sqrt{q_0}}{|x|^2 \log |x|}  J_1 (-2 \sqrt{q_0} \log|x|).
\end{equation}
Using the well-known asymptotics for the Bessel functions at infinity, (see \cite[Eq. (5.11.6)]{Lebedev} we see that: 
\begin{equation} \label{id:Bornpotential2asympt}
\ve(x) = O\left(\frac{1}{|x|^2 |\log|x||^{\frac{3}{2}}}\right) ,\quad  |x| \to 0.
\end{equation}

In particular, we see that the singularity at the origin for the potential $\ve$ is more or less the same as the one for the initial potential $V$.

In the same way, if  $q_0 <0$, we obtain:
\begin{equation} \label{id:Bornpotential3}
\ve(x) = - \frac{\sqrt{-q_0}}{|x|^2 \log |x|}  I_1 (-2 \sqrt{-q_0} \log|x|),
\end{equation}
where $I_1$ is the corresponding modified Bessel function of order one, and we have the following asymptotics (see \cite[Eq. (5.11.10)]{Lebedev}):
\begin{equation} \label{id:Bornpotential3asympt}
\ve(x) = O\left(\frac{1}{|x|^{2+2\sqrt{-q_0} } |\log|x||^{\frac{3}{2}}}\right) , \quad |x| \to 0.
\end{equation}
So, in this case the singularity at $x=0$ for $\ve$ is stronger than the one for $V$.


\section{Effective reconstruction algorithms}  \label{sec:reconstruction}

The proof of \Cref{thm:simon_uniqueness} in \cite{IST1}  is close to constituting an explicit reconstruction algorithm for the potential $Q$ in terms of the $A$-amplitude. In this section we show that this approach can be adapted to the Calderón problem, and together with formula \eqref{id:regularization}, yields \Cref{remark:reconstruction}, a method to reconstruct  a radial potential $V$ from its Cauchy Data or DtN map.

\subsection{Simon's approach to reconstruction} \label{sec:subsec_reconstruction_simon} \
The idea introduced by Simon in \cite{IST1} is to study the $A$-amplitudes of the translated  potentials $Q_s(t) = Q(t+s)$. Notice that this removes a part of $Q$ from the domain $\R_+$, since we are translating  $Q$ to the left. The key is that, by \Cref{thm:simon_avdonin_error}, one can read the value of $Q_s(0) = Q(s)$ from the corresponding  $A$-amplitudes of $Q_s$ as the potential is translated out of the domain. It also will be essential to use that the $A$-amplitudes of $Q_s$ are related by a certain equation.

Let $Q : \R_+ \to \R$ be a potential satisfying assumption \eqref{id:assumption_A2}.
Let $A$ be the $A$-amplitude associated to $Q$. For every fixed $s \ge 0$, consider the potential
$Q_s(t) = Q(t+s)$, where $t\in \R_+$, and denote by $A(t,s)$ the corresponding $A$-amplitude of $Q_s$. Since $Q$ has just local $L^1$ regularity, and so does $A$, it is not clear if $A(t,s)$ is well defined. The simplest way overcome this difficulty, given a specific realization of $Q$,  is to pick the  realization of $A(t,s)$ such that $A(t,s) -Q_s(t)$ is a continuous function for every fixed $s$. This is enough to properly define $A(t,s)$, as the following lemma shows.
\begin{lemma} \label{lemma:Ats_Qts}
Assume that $Q$ satisfies \eqref{id:assumption_A2}. Then, $ A(t,s) -Q(t+s)$ is a jointly continuous function on $[0,\infty) \times [0,\infty)$, and for all $t,s\in \R_+$ it holds that
\begin{equation*}
    \left| A(t,s) -Q(t+s) \right| \le \alpha(t,s)^2 e^{t \, \alpha(t,s) },
\end{equation*}
where
\begin{equation} \label{e:alpha_def}
     \alpha(t,s) =  \int_{0}^t |Q(y+s)| \, dy  = \int_{s}^{t+s} |Q(y)| \, dy .
\end{equation}
\end{lemma}
The estimate follows  applying \Cref{thm:simon_avdonin_error} to the potential $Q_s(t)$. We postpone momentarily the proof of the continuity statement.

Let $Q\in \cC^1(\R_+)$ satisfying \eqref{id:assumption_A2}. In \cite{IST1} it is shown that 
$A(t,s)$ satisfies the initial value problem
\begin{equation} \label{id:A_pde_simon}
\begin{aligned}
     &\frac{\partial A}{\partial s}(t,s)   = \frac{\partial A}{\partial t}(t,s) + \int_0^t A(w,s) A(t-w,s) \, dw,\qquad (t,s) \in \R_+ \times \R_+, \\
 &  A(t,0) = A(t),   \qquad t\in \R_+,
\end{aligned}
\end{equation}
where $A(t)$ denotes the $A$-amplitude of $Q$.
If   $Q \in \cC^1(\R_+)$ this equation holds in the strong sense, and also in the general case under a suitable weak formulation   (see \Cref{thm:Simon_integral_equation}). Then, it follows from \Cref{thm:simon_avdonin_error} that
\begin{equation} \label{id:lim_A-Q}
    \lim_{t \to 0^+} A(t,s) = Q(s) ,
\end{equation}
where the convergence  holds in $L^1(0,T)$ for all $T>0$. If $Q$ is continuous, then the convergence holds also point-wise, and in general will hold at any point of right Lebesgue continuity of $Q$ (see \cite{IST1}). Therefore \eqref{id:lim_A-Q} together with \eqref{id:A_pde_simon}   give a procedure  to reconstruct the potential $Q$ from its $A$-amplitude, provided that \eqref{id:A_pde_simon} can be uniquely solved under certain assumptions. This will be proved in \Cref{lemma:uniqueness_non-linear_A} below, and
in \Cref{sec:stability} we will analyze the stability of this reconstruction procedure. We start by stating a weak version  of \Cref{id:A_pde_simon}.

\begin{theorem}[Theorem 6.3 of  \cite{IST1}] \label{thm:Simon_integral_equation}
Let $Q$ such that  \eqref{id:assumption_A2} holds. \\ If $K(t,s) = A(t-s,s)$ then
\begin{equation} \label{id:K_eq}
      K(t,s_2) = K(t,s_1) 
+   \int_{s_1}^{s_2} \int_{y_2}^{t} K(y_1 ,y_2) K \left (t-y_1+y_2,y_2\right ) dy_1  dy_2, 
\end{equation}
with $0<s_1<s_2<t<\infty$.
\end{theorem}
\begin{remark} \label{remark:continuous}
The previous theorem implies that $K$ satisfies an initial value problem with $K(t,0) = A(t,0) =A(t)$, where $A(t)$ is the $A$-amplitude of $Q$. Moreover, by \Cref{lemma:Ats_Qts} we know that $K(t,s)-A(t) = A(t-s,s)-Q(t) - (A(t)-Q(t))$ is  continuous for $0\le s \le t<\infty$.
\end{remark}
The previous conditions are enough to obtain a uniqueness result for the initial value problem for \eqref{id:K_eq}, as the following lemma shows.

\begin{lemma}   \label{lemma:uniqueness_non-linear_A}
Let $a \in \R_+$ and $f \in L^1(0,a)$. There is at most one solution of \eqref{id:K_eq} in $\{0<s_1<s_2<t<a\}$
such that 
\[
 K(t,s) = f(t) + K_0(t,s)  ,
\]
where $ K_0(t,s)$ is a continuous function on $0<s\le t \le a$ and $K_0(t,0) = 0$.
\end{lemma}
This lemma is a consequence of the estimates proved by Simon in \cite[Section 7]{IST1}. Since the result is not explicitly stated in \cite{IST1}, we give a proof here for completeness.
\begin{proof} 
Assume that $ K(t,s) = f(t) + K_0(t,s)$ and $ \tilde{K}(t,s) = f(t) + \tilde{K}_0(t,s)$ are two solutions satisfying the conditions of the statement of the lemma. 
Let
\[
g(s) := \int_{s}^{a} \left| K(t,s) -\tilde{K}(t,s) \right| \, dt =\int_{s}^{a} \left| K_0(t,s) -\tilde{K}_0(t,s) \right| \, dt.
\]
By the previous assumptions, $g$ is a continuous function in $[0,a]$. 
Moreover
\begin{multline} \label{est:D_finite}
D := \sup_{0\le s<a}  \int_{s}^{a} \left( \left| K(t,s) \right| + \left| \tilde{K}(t,s) \right|  \right) \, dt 
\\ < \sup_{0\le s<a} \int_{s}^{a} \left( \left| K_0(t,s) \right| + \left| \tilde{K_0}(t,s) \right|  + 2\left|f(t) \right|  \right) \, dt< \infty,
\end{multline}
since the function of $s$ obtained from last integral  is  continuous on $[0,a]$.
Using this in \eqref{id:K_eq} it follows that
\[
g(s_2) \le g(s_1) + D \int_{s_1}^{s_2} g(y) \, dy,
\]
(this follows from simple estimates from \eqref{id:K_eq}: for a completely analogous and more detailed explanation see the proof of \Cref{lemma:gronwall} below).

We now define $h_{z}(s) = \sup_{z \le y  \le s} g(y)$. The previous estimate implies
that 
\[
h_{s_1}(s_2) \le h_{s_1}(s_1) + D h_{s_1}(s_2) \int_{s_1}^{s_2}  \, dy.
\]
Therefore,  if $h_{s_1}(s_1) = g(s_1) = 0$ and $(s_2-s_1) D<1$, then  it follows that $h_{s_1}(s_2) =0$. This shows that if $g(s_1)$ vanishes, then $g(s)$ vanishes in $(s_1,s_1+1/D)$. 

Since $g(0) =0$ one can apply the previous argument a finite number of times to deduce that $g(s) =0$ in $[0,a]$, and therefore that $K = \tilde{K}$ on $\{0<s_1<s_2<t<a\}$.
\end{proof}

We can now prove that $A(t,s)-Q(t+s)$ is a jointly continuous function, as stated previously.

\begin{proof}[Proof of $\Cref{lemma:Ats_Qts}$]
The estimate is a direct application of \eqref{est:a_Q_estimate_2}. 

We now prove that $A(t,s) -Q(t+s)$ is continuous in $t$ and $s$. We define
\[
\alpha(t) = \int_0^t |Q_1(s)| + |Q_2(s)| \, ds.
\]
From \cite[Theorem 2.1]{IST1} for $L^1(\R_+)$ potentials\footnote{Note that \cite[equation (2.4)]{IST1} contains a typographical error, the correct left hand side is the one in \eqref{est:A_difference} instead of just $|A_1(t)-A_2(t)|$.} one gets the estimate
\begin{equation} \label{est:A_difference} 
   \left |A_1(t) -Q_1(t) - \left( A_2(t) -Q_2(t) \right) \right | \le \alpha(t) e^{ t\alpha(t)}  \int_{0}^t |Q_1(s)-Q_2(s)|\, ds .
\end{equation}
The extension of this estimate for potentials satisfying \eqref{id:assumption_A2} is immediate due to the local dependence  of the $A$-amplitude from $Q$ (see \Cref{thm:simon_uniqueness}). 

This shows that $A-Q$ is continuous with respect to $Q$ in $L^1(0,T)$  for all $T>0$.

Let $T>0$ and $s_1,s_2 \in [0,T)$ with $s_1<s_2$. We now apply \eqref{est:A_difference}  with $Q_1(t) = Q(t+s_1)$ and $Q_2(t) = Q(t+s_2)$. With this choice we have $A_1(t) = A(t,s_1)$, $A_2(t) = A(t,s_2)$ and
\begin{multline} \label{est:A_difference_2}
   \sup_{t\in[0,T]} \left |A(t,s_1) -Q(t+s_1) - \left( A(t,s_2) -Q(t+s_2) \right) \right | \\
   \le \alpha(2T) e^{ 2T\alpha(2T)}  \int_{0}^{2T} |Q(s)-Q(s+s_2-s_1)|\, ds ,
\end{multline}
where we have used a change of variable $s = s'+s_1$ in the integral term.
Since translations are continuous in the $L^1$ norm, we have that $\lim_{\varepsilon \to 0^+}\omega(\varepsilon) =0$ where
\[
\omega(\varepsilon) = \int_{0}^{2T} |Q(s)-Q(s+\varepsilon)|\, ds.
\]
From \Cref{thm:simon_avdonin_error} it follows that for any fixed $s \in [0,T)$ the function $A(t,s) -Q(t+s)$ is  continuous in $t$  for  $t\in [0,T]$. Combining this with the estimate
\begin{equation*}
   \sup_{t\in[0,T]} \left |A(t,s_1) -Q(t+s_1) - \left( A(t,s_2) -Q(t+s_2) \right) \right | \le C_T \omega(s_2-s_1),
\end{equation*}
that follows from  \eqref{est:A_difference_2},
we obtain that $A(t,s) -Q(t+s)$ is a jointly continuous function in $[0,T)^2$. Since $T$ is arbitrary, this finishes  the proof of the lemma.
\end{proof}

\subsection{Reconstruction for the radial Calderón problem}\
It is not difficult to adapt the prevous reconstruction method to the radial Calderón problem
using the transformation $V(x) = q(|x|) = |x|^{-2}Q(-\log|x|)$ for the potentials, as we now show.

Let $V(x) = q(|x|)$ and define
 \begin{equation*}
     V_s(x) := s^2V(sx), \quad  q_s(r) := s^2q(sr)\quad  s\in (0,1].
 \end{equation*} 
If $ \ve_s : = \left[ V_s \right]^\mB $  we introduce the $W$ function given by
 \begin{equation} \label{id:W_def}
     W(|x|,s) := \frac{1}{s^2} \ve_s (x), \quad s\in [0,1].
 \end{equation}
It will be convenient to use the notation 
 \begin{equation*} 
\ve(x) = \qe(|x|), \quad  \text{and} \quad \ve_s(x) = \qe_s(|x|),
\end{equation*}
for the radial profiles of the Born approximations. Thus  $W(r,s) = s^{-2} \qe_s (r)$.
 
 In terms of the $A$-amplitude it holds that  
\begin{equation} \label{id:W_to_A}
    A(t,s) = e^{-2(t+s)}W(e^{-t},e^{-s})  , \qquad   W(r,s) = \frac{1}{r^2s^2} A(-\log r,-\log s).
\end{equation}

We restate \Cref{lemma:Ats_Qts} in this context as follows.
\begin{lemma} \label{lemma:Ats_Qts_V}
Assume that $V = q(|\cdot|)$ with $V\in \cV_d$. Then,   for all $r,s\in (0,1]$ it holds that
\begin{equation*}
    \left| F_s(r) \right| \le  \frac{s^{-2}}{r^{
2+ g(r,s)}} g(r,s)^2,
\end{equation*}
where  
\[
F_s(r) := W(r,s) - q(rs), \quad \text{and} \quad g(r,s) := \int_r^1 w s^2|q(sw)|   \, dw  = \int_{rs}^s t |q(t)|   \, dt .
\]
Moreover $F_s(r)$ is a jointly continuous function on $(0,1] \times (0,1]$.
\end{lemma}
The proof is straightforward using \eqref{id:W_to_A} and  \Cref{lemma:Ats_Qts}. We can now prove the analogue of \eqref{id:lim_A-Q}.
  
\begin{proposition}  \label{prop:lim_W_to_V}
Assume that $V\in \cV_d$ with $V= q(|\cdot|)$ and fix $b$ such that $0<b<1$. Then, if $W$ is given by \eqref{id:W_def} it holds that
\begin{equation*}
    \lim_{r\to 1^-} W(r,\cdot) =  q(\cdot)  , \quad \text{in } L^1(b,1). 
\end{equation*}
\end{proposition}
\begin{proof}
We use that
\begin{align*}
       W(r,s)  &=    q(s) + (W(r,s)-q(sr)) +  (q(sr)-q(s))\\
                &=    q(s) + F_s(r) +  (q(sr)-q(s)),
\end{align*}
where $F_s(r) = W(r,s) - q(rs)$.
From \Cref{lemma:Ats_Qts_V} and a simple bound for $g(r,s)$ it follows that
\[
\left|F_{s}(r) \right| \le \frac{b^{-2}}{r^{h(r)+2}} \left(\int_{rs}^s t |q(t)| \, dt \right)^2, \quad \text{with} \quad h(r)  = \int_{br}^1 t |q(t)|   \, dt.
\]
Hence $\lim_{r \to 1^-} \int_{b}^1 |F_s(r)| \, ds =0$ by dominated convergence. It remains to prove that
\begin{equation*}
    \lim_{r\to 1^{-}} \int_b^1 |q(sr)-q(s)| \, ds =0,
\end{equation*}
which follows from the continuity of dilations in the $L^1$ norm.
\end{proof}

Using the transformation \eqref{id:W_to_A} in \eqref{id:A_pde_simon}, one can   show that $W$ satisfies also a first order PDE with a non-linear integral term:
\begin{equation} \label{id:W_pde}
     r\frac{\partial W}{\partial r} (r,s) -s\frac{\partial W}{\partial s} (r,s)   =  s^2\int_{r}^1 W\left (\frac{r}{\nu},s\right)W(\nu,s) \frac{d\nu}{\nu} ,
\end{equation}
for all $(r,s) \in (0,1) \times (0,1) $. This holds in the classical sense for $\cC^1(\IB^d)$ potentials, since in this case  $A(t,s)$---and hence $W(r,s)$---is a jointly  $\cC^1$ function,  as shown in \cite[Section 2]{IST1}. 

With  the change variables $U(r,s) = W \left( \frac{r}{s},s \right)$ the equation \eqref{id:W_pde}  becomes
\begin{equation*} 
      \frac{\partial U}{\partial s} (r,s)  =   -s\int_{r}^{s} U(\nu,s) U \left (r\frac{s}{\nu},s\right ) \frac{d \nu }{\nu } ,   \qquad 0<r<s<1.
\end{equation*}

In the general case  $V \notin \cC^1(\IB^d)$,  $W(r,s)$ can be shown to satisfy the  integral version of the previous equation.
\begin{proposition} \label{prop:uniqueness_U_sol}
Let $V\in \cV_d$ and let $W$ be given by \eqref{id:W_def}. Define
\begin{equation*}
    U(r,s) := W \left( \frac{r}{s},s\right), \qquad 0<r<s<1  .
\end{equation*}
 Then, we have that
\begin{equation} \label{id:W_pde_2}
     U(r,s_2)   =  U(r,s_1)
+  \int_{s_2}^{s_1} y_2 \int_{r}^{y_2}  U(y_1 ,y_2) U \left (r\frac{y_2}{y_1},y_2\right ) \, \frac{dy_1}{y_1}  \, dy_2,
\end{equation}
for all $0<r<s_2<s_1<1$.

In addition, it holds that $U(r,s)$ is the unique solution of \eqref{id:W_pde_2} in $0<r<s<1 $ such that:
\begin{enumerate}[i)]
    \item  $U_0(r,s) := U(r,s) - U(r,1) $ is a continuous function for $ 0 < r\le s \le 1 $ and $U_0(r,1) =0$.
    \item  $U(r,1) = \qe(r)$ for $0<r<1$.
\end{enumerate}
\end{proposition} 
The equation \eqref{id:W_pde_2} has a strong local behavior even if it contains a non-local term: the value $U(r_0,s_0)$ of a solution   only depends on the values of $U$ in the triangular  region
\[
D_{(r_0,s_0)} = \{ (r,s) \in (0,1)^2: r\le s, \, r \ge r_0, \, s \ge s_0  \}.
\]
To see this notice that taking $s_2=s_0$ and $r = r_0$ in the integral term in \eqref{id:W_pde_2} we have $y_2\ge s_0$, $y_1\ge r_0$ and $r_0\frac{y_2}{y_1} \ge r_0$. This gives the equation a   local behaviour that is in turn reflected   in \Cref{mt:uniqueness} and other results.

The proof of \Cref{prop:uniqueness_U_sol} is based on the  fact that the initial value problem for  \eqref{id:W_pde_2} has at most one  solution that is a continuous perturbation of a free solution, as the next lemma states.
\begin{lemma}   \label{lemma:uniqueness_non-linear}
Let $b\in (0,1)$ and $f \in L^1(b,1)$. There is at most one solution of
\begin{equation*} 
     U(r,s_2)   =  U(r,s_1)
+  \int_{s_2}^{s_1} y_2 \int_{r}^{y_2}  U(y_1 ,y_2) U \left (r\frac{y_2}{y_1},y_2\right ) \, \frac{dy_1}{y_1}  \, dy_2,
\end{equation*}
with   $b<r<s_2<s_1<1$ such that 
\[
 U(r,s) = f(r) + U_0(r,s)  ,
\]
where $ U_0(r,s)$ is a continuous function for $b<r\le  s \le 1$ and $U_0(r,1) = 0$.
\end{lemma}
\begin{proof}
    Is an immediate consequence of  \Cref{lemma:uniqueness_non-linear_A} and \eqref{id:W_to_A}.
    which implies that 
\[
 U(r,s) = W\left (\frac{r}{s},s\right ) = \frac{1}{r^2s^2} A(-\log(r)+\log s,-\log s) = \frac{1}{r^2s^2} K(-\log r,-\log s). 
\]
\end{proof}

\begin{proof}[Proof of $\Cref{prop:uniqueness_U_sol}$] 
That  \eqref{id:W_pde_2} holds for all $V\in \cV_d$  follows directly from \eqref{id:W_to_A} and \Cref{thm:Simon_integral_equation}. 

The second statement follows from \Cref{lemma:uniqueness_non-linear}, provided that we show that $U_0(r,s) = W \left( \frac{r}{s},s \right) -  W \left( r,1 \right)$ is a continuous function on $\{0<r\le  s \le 1\}$. We have that
\begin{equation}  \label{id:U_0_continuous}
\begin{aligned}
    U_0(sr,s) &=   W \left( r,s \right) -  W \left( rs,1 \right) \\
            &=  F_s(r) - \left( q^\mB(rs) -q(rs) \right ) .
\end{aligned}
\end{equation}
By \Cref{lemma:Ats_Qts_V} we know that $F_s(r)$ is continuous on $(0,1]^2$. On the other hand,   $q^\mB(r') -q(r') = F_1(r')$ is continuous on $(0,1]$, so, taking $r'=rs$, the second term in \eqref{id:U_0_continuous} is also continuous on $(0,1]^2$. Replacing $r$ by $r/s$ in \eqref{id:U_0_continuous} we conclude that $ U_0(r,s)$ is a continuous function on $\{0<r\le  s \le 1\}$.
\end{proof}

We can finally state the algorithm to reconstruct $V\in\cV_d$ form $\ve$.
\begin{algorithm} \label{remark:reconstruction}
Given $\ve$ and $0<b<1$, it is possible to reconstruct $V\in\cV_d$ in the region $b<|x|<1$ with the following three steps:
\begin{enumerate}[1)]
    \item  Using \eqref{id:VB_regularization_intro} and that $\ve = \ve_r$ on $\IB^d\setminus\{0\}$,   reconstruct $\ve$ from  $(\lambda_k[V])_{k\in \N_0}$. 

\item Find the unique solution $U(r,s)$ of \eqref{id:W_pde_2} such that $U_0(r,s) = U(r,s) - q^\mB(r)$ is a continuous function with $U_0(r,1) =0$ and $q^\mB(|x|) = \ve(x)$.

\item  Use that
$
    \lim_{r\to 1^-} U(r|x|,|x|) =  V(x)
$ where the convergence is in $L^1(\{b<|x|<1\})$.
\end{enumerate}
\end{algorithm}
We notice that step 1) might be replaced by any other suitable method to solve the moment problem \eqref{id:thm_1_1}.

\section{Global Hölder stability} \label{sec:stability}

\subsection{Stability of the \texorpdfstring{$A$}{A}-amplitude}

Stability results for inverse spectral problems for Schrödinger operators on the half-line go back to \cite{MarMas70} (see also \cite{Marbook}); there is also a vast literature on similar results in the case of Sturm-Liouville operators on a finite interval, see, for instance, \cite{Alek, HoKiss, SavShka10} and the references therein.

The main result in this section explores this type of stability result when the spectral data are expressed in terms of the $A$-amplitude. Our next result proves that the map $A\longmapsto Q$ is Hölder continuous.
\begin{theorem}  \label{thm:stability_A_function}
 Fix $M>1$,  $a \in (0,\infty]$,  $\varepsilon_0:=\min\left(1,a\right)$, and $1<p\le \infty$.
 Consider two potentials $Q_1$, $Q_2$  satisfying assumption  \eqref{id:assumption_A2} and let  $A_j$ be the $A$ function of $Q_j$.   Assume also that
 \begin{equation} \label{id:def_Q_loc_bound}
\max_{j=1,2}  \left(\norm{Q_j}_{L^1(0,a)} + \norm{Q_j}_{L^p(0,a)}  \right) \le   \frac{1}{2} M .
\end{equation}
and that
\begin{equation} \label{est:A_condition}
  \int_0^a  \left| A_1(t) - A_2(t) \right|e^{-Mt}  \, dt <\varepsilon_0^{(1+p')/p'},
\end{equation}
where $p'$ is its Hölder conjugate exponent of $p$. Then, for $0<a<\infty$ it holds that
 \begin{equation} \label{e:stab_A_loc}
      \int_0^{a}  \left| Q_1(t)-Q_2(t) \right|  \,dt    < e^{aM}\left(e^{2Ma  }  + 4M  \right)
     \left(\int_0^a  \left| A_1(t) - A_2(t) \right|  \, dt\right)^{1/(1+p')}.
\end{equation}
And if $a=\infty$ it holds that
 \begin{equation} \label{is:Asta_global}
      \int_0^{\infty}  \left| Q_1(t)-Q_2(t) \right| e^{-Mt} \,dt    < ( e^{2M} +5M )
     \left(\int_0^\infty  \left| A_1(t) - A_2(t) \right|e^{-Mt} \, dt\right)^{\beta},
\end{equation}
where $\beta =  (1+p'(1+2M))^{-1}$.
 \end{theorem}
\begin{remark} \label{remark:finitness_A}
Under the assumptions in the statement of the theorem, the right hand side of \eqref{is:Asta_global} is always finite, so the estimate is non-trivial. This follows from \Cref{l:Da} below.
\end{remark}

Before proving this theorem, we will need some preliminary results.
For every fixed $s \ge 0$, consider the potential
$Q_s(t) = Q(t+s)$, where $t\in \R_+$, and denote by $A(t,s)$ the $A$-amplitude of $Q_s$. 
 
If $s\ge 0$, denote by $A_1(\cdot,s)$ and $A_2(\cdot,s)$ the corresponding $A$-amplitudes of the translated potentials $Q_1(\cdot +s)$ and $Q_2(\cdot +s)$.
Also, for $a >0$ we define
\begin{equation} \label{id:def_g_a}
    g_a(s) := \int_0^{a-s} \left| A_2(t,s) - A_1(t,s) \right| e^{-M(t+s)} \, dt, \quad s \in [0,a],
\end{equation}
and the constant
\begin{equation} \label{id:def_Da}
    D_M(a) := \sup_{0\le s < a} \int_0^{a-s}  \left[ \left| A_1(t,s)  \right| + \left| A_2(t,s)  \right| \right] e^{-Mt} \, dt.
\end{equation}

\begin{lemma} \label{l:Da}
    The following holds:
    \begin{enumerate}[i)]
        \item   Let $Q \in L^1(\R_+)$. For any constant $M \ge  2  \norm{Q}_{L^1(\R_+)}$, 
        \begin{equation*} 
            \int_0^\infty |A(t)| e^{-Mt} \, dt \le 2\norm{Q}_{L^1(\R_+)}. 
        \end{equation*}
        \item   Let $Q_j$, $j=1,2$ satisfying \eqref{id:assumption_A2} and $0<a \le \infty$. For any constant $M \ge  2\max_{j=1,2}   \norm{Q_j}_{L^1(0,a)}$, it holds that $D_M(a) \le 2M$.
    \end{enumerate}
\end{lemma}
\begin{proof}
    By \eqref{est:a_Q_estimate_2}, (or \Cref{lemma:Ats_Qts} with $s=0$) one has
    \begin{equation*}
    |A(t) | \le \norm{Q}_{L^1(\R_+)}^2 e^{t \norm{Q}_{L^1(\R_+)} } + |Q(t)| ,
    \end{equation*}
    and therefore, since $M-\norm{Q}_{L^1(\R_+)} \ge \norm{Q}_{L^1(\R_+)}$, one obtains
\begin{equation*}
    \int_0^\infty |A(t) | e^{-Mt}\, dt \le \norm{Q}_{L^1(\R_+)}^2 \int_0^\infty  e^{-t\norm{Q}_{L^1(\R_+)}}  \, dt + \norm{Q}_{L^1(\R_+)} \le  2\norm{Q}_{L^1(\R_+)}.
\end{equation*}
    This proves i). Analogously, by \Cref{lemma:Ats_Qts} we have
    \begin{multline*}
        \sup_{0\le s < a} \int_0^{a-s}  | A_j(t,s) | e^{-Mt} \, dt \le \sup_{0\le s < a} \int_0^{a-s}  | Q_j(t+s) | e^{-Mt} \, dt \\
        +\sup_{0\le s < a} \int_0^{a-s}  \alpha_j(t,s)^2  e^{t(\alpha_j(t,s)-M)} \, dt  \le  \frac{1}{2}M +\frac{1}{4}M^2 \int_0^\infty  e^{-Mt/2}  \, dt   \le M,
    \end{multline*}
    since $\alpha_j(t,s) \le \norm{Q_j}_{L^1(0,a)} \le \frac{1}{2}M$.
\end{proof}

\begin{lemma}  \label{lemma:gronwall}  
Let  $Q_1,Q_2$ satisfying \eqref{id:assumption_A2}, and fix $0<a<\infty$. Then $g_a(s) $ is a continuous function on $[0,a]$.

In addition, if $M \ge  2\max_{j=1,2}   \norm{Q_j}_{L^1(0,a)}$ it holds that $ g_a(s) \le g_a(0)   e^{2Ms} $.
\end{lemma}
\begin{proof}
    Define 
    \begin{equation*}
        E_j(t,s) := A_j(t-s,s) e^{-Mt},\qquad 0\le s\le t\le a, \quad j=1,2.
    \end{equation*}  
    Since
    \begin{equation*}
        e^{Mt}\left(E_j(t,s) - A_j(t)e^{-Mt}\right) =A_j(t-s,s)-Q_j(t) - \left( A_j(t)-Q_j(t) \right),
    \end{equation*}
    it follows from \Cref{lemma:Ats_Qts} that $E_j(t,s) - A_j(t)e^{-Mt}$ is a continuous function, and therefore, that $g_a(s)$ is continuous on $[0,a]$.
    
    By a simple change of variables, we have
    \begin{equation}  \label{e:g_K}
       g_a(s) =  \int_s^{a} \left| E_2(t,s) - E_1(t,s) \right|  \, dt , \quad s \in [0,a].
    \end{equation}
    and
    \begin{equation}  \label{e:g_Da}
    D_M(a)= \sup_{0\le s < a} e^{Ms}\int_s^{a}  \left( \left| E_1(t,s)  \right| + \left| E_2(t,s)  \right| \right)   \, dt.
    \end{equation}
    
    On the other hand, $K(t,s)= E_j(t,s)e^{Mt}$  satisfies \eqref{id:K_eq} for $j=1,2$, so if $s_1<s_2<a$ we have
    \begin{equation*} 
          E_j(t,s_2) = E_j(t,s_1) 
    +   \int_{s_1}^{s_2} e^{My_2} \int_{y_2}^{t} E_j(y_1 ,y_2) E_j \left (t-y_1+y_2,y_2\right ) dy_1  dy_2.
    \end{equation*}
    We take the difference of the previous identity for $j=1$ and $j=2$, and integrate  in $t$ in the interval $[s_2,a]$ and use that $s_1<s_2$ to obtain the estimate
    \begin{multline*}
      \int_{s_2}^a \left|E_1(t,s_2)-E_2(t,s_2)\right|  \, dt \le \int_{s_1} ^a \left|E_1(t,s_1) -E_2(t,s_1) \right|   \, dt \\
     +   \int_{s_1}^{s_2} e^{My_2} \int_{s_2}^a \int_{y_2}^{t} \left |\sum_{j=1}^2 (-1)^{j+1} E_j(y_1 ,y_2)E_j \left (t-y_1+y_2,y_2\right ) \right|\, dy_1  dt  dy_2.
    \end{multline*}
    Using \eqref{e:g_K} and adding and subtracting the corresponding crossed term $E_1E_2$ we get
    \begin{equation}  \label{e:g_a_I_est}
      g_a(s_2)    \le g_a(s_1) + \int_{s_1}^{s_2} e^{M s}  \left(I(E_1-E_2,E_1;s) + I(E_2,E_1-E_2;s) \right) \, ds ,
    \end{equation} 
    where
    \begin{equation*}
        I(F,G;s) :=  \int_{s_2}^a \int_{s}^{t} | F(y_1 ,s)|| G (t-y_1+ s,s )|    \, dy_1  dt.
    \end{equation*}
    Changing the order of integration one gets the estimate
    \begin{equation*}
        I(F,G;s) \le       \int_{s}^a | F(t ,s)| \, dt \int_{s}^{a} | G (t, s)|  \, dt.
    \end{equation*}
    Hence, we obtain
    \begin{equation*}
        I(E_1-E_2,E_1;s) + I(E_2,E_1-E_2;s) 
        \le  g_a(s) \int_s^{a}  \left( \left| E_1(t,s)  \right| + \left| E_2(t,s)  \right| \right)   \, dt.
    \end{equation*}
    Inserting this estimate in \eqref{e:g_a_I_est} and using \eqref{e:g_Da} yields
    \begin{equation*} 
        g_a(s_2) \le g_a(s_1) + D_M(a)\int_{s_1}^{s_2} g_a(y) \, dy, \qquad  s_1 <s_2<a.
    \end{equation*}
    Therefore for $s_1 =0$ and $s_2=s$  it reduces to
    \begin{equation*}  
        g_a(s) \le g_a(0) + D_M(a)\int_{0}^{s} g_a(y) \, dy.
    \end{equation*}
    since $g_a$ is continuous, a direct application of   Grönwall's inequality proves that
    \begin{equation*} 
        g_a(s) \le  
          g_a(0)   e^{s D_M(a)  }  ,
    \end{equation*}
    so, the estimate  $g_a(s) \le g_a(0)   e^{2Ms} $ follows by \Cref{l:Da} ii).
\end{proof}

We can now prove the main stability estimates.
 
\begin{proof}[Proof of \Cref{thm:stability_A_function}]
Let $s,t \ge 0$. We start by assuming $0<a<\infty$. We have that
\begin{multline*}
  Q_1(t+s) -Q_2(t+s) =  \\ Q_1(t+s) -A_1(t,s) - 
  \left(Q_2(t+s)   - A_2(t,s )  \right) + \left(A_1(t,s )-A_2(t,s )\right).
\end{multline*}
Thus, for  $0<\varepsilon <\varepsilon_0 = \min(1,a)$  and for all $s \ge 0$:
\begin{multline} \label{est:initial}
    \int_0^\varepsilon \left| Q_1(t+s)-Q_2(t+s) \right| e^{-M(t+s)} \, dt  \le  
    \int_0^{\varepsilon }\left| A_1(t,s)-A_2(t,s) \right|e^{-M(t+s)} \, dt \\
+\int_0^\varepsilon \left| A_1(t,s)-Q_1(t+s) \right| e^{-M(t+s)} \, dt  +\int_0^\varepsilon \left|A_2(t,s)-Q_2(t+s) \right| e^{-M(t+s)} \, dt .
 \end{multline}
We now assume that $0\le s<a-\varepsilon$. By \eqref{id:def_g_a}, the first term on the right  satisfies 
\begin{equation}  \label{est:ga}
     \int_0^{s+\varepsilon-s } \left| A_1(t,s)-A_2(t,s) \right| e^{-M(t+s)} \, dt = g_{s +\varepsilon}(s) \le g_a(s),
\end{equation}
since $s+\varepsilon <s+\varepsilon_0 <a$. If $j=1,2$, applying \Cref{lemma:Ats_Qts} with $Q = Q_j$,   the remaining terms satisfy
\begin{multline} \label{est:A_Q_bound_beta}
  \int_0^\varepsilon \left| A_j(t,s)-Q_j(t+s) \right| e^{-M(t+s)}\, dt    
  \le  \int_0^\varepsilon  \alpha_j(t,s)^2 e^{t \alpha_j(t,s) } e^{-M(t+s)} \, dt \\
   \le  e^{-Ms}\int_0^\varepsilon  \alpha_j(t,s)^2 e^{t (\alpha_j(t,s)-M) }  \, dt
  \le \varepsilon  \alpha_j(\varepsilon,s)^2 e^{-Ms}   ,
\end{multline} 
since $\alpha_j(t,s) \le 2^{-1}\norm{Q_j}_{L^1(0,a)}\le 2^{-1}M$, and $0\le t\le \varepsilon$ implies $\alpha_j(t,s) \le  \alpha_j(\varepsilon,s)$ (Here $\alpha_j$ is given by \eqref{e:alpha_def} with $Q=Q_j$). By assumption \eqref{id:def_Q_loc_bound} and   Hölder inequality, together with the fact that $s+\varepsilon<a$,  we have
\begin{equation*}
    \alpha_j(\varepsilon,s) = \int_{s}^{s+\varepsilon} |Q_j(y)| \, dy \le \varepsilon^{1/p'} \norm{Q_j}_{L^p(0,a)} \le \varepsilon^{1/p'}M,
\end{equation*}
where $1\le p'<\infty$ is the conjugate exponent of $p$.
Then \eqref{est:A_Q_bound_beta}  becomes
\begin{equation} \label{est:ma}
  \int_0^\varepsilon \left| A_j(t,s)-Q_j(t+s) \right| e^{-M(t+s)}\, dt   \le \varepsilon^{1+2/p'} M^2 e^{-Ms}.
\end{equation}
Inserting \eqref{est:ga} and \eqref{est:ma} with $j=1,2$ in \eqref{est:initial} gives 
\begin{equation*}
     \int_0^\varepsilon \left| Q_1(t+s)-Q_2(t+s) \right|e^{-M(t+s)} \, dt  \le
      g_{a}(s)     + 
2  \varepsilon^{1+2/p'} M^2 e^{-Ms}    ,
\end{equation*}
and \Cref{lemma:gronwall} to bound $g_a(s)$ yields
\begin{equation*}
     \int_0^\varepsilon \left| Q_1(t+s)-Q_2(t+s) \right| e^{-M(t+s)} \, dt  \le
      g_a(0)   e^{ 2M s } + 
2  \varepsilon^{1+2/p'} M^2 e^{-Ms}  .
\end{equation*}
We now integrate both sides in the $s$ variable:
\begin{align*}
    \int_0^\varepsilon  \int_0^{a-\varepsilon}  &\left| Q_1(t+s)-Q_2(t+s) \right|e^{-M(t+s)} \, ds \, dt \\
    &\le   \frac{ g_a(0)}{2M}\left( e^{2M(a-\varepsilon)   } -1\right) + 
    2    \varepsilon^{1+2/p'} M (1-e^{-(a-\varepsilon)M}) \\
    & \le  (2M)^{-1}e^{2Ma  } g_a(0)    + 2M \varepsilon^{1+2/p'} ,
\end{align*}
which, changing variables in the first integral and using that $(2M)^{-1}<1$, immediately implies
 \begin{equation} \label{est:stab_final_1}
    \int_0^\varepsilon  \int_t^{a+t-\varepsilon}  \left| Q_1(s)-Q_2(s) \right|e^{-Ms} \,ds \, dt  \le  
  e^{2Ma  } g_a(0)    + 2M \varepsilon^{1+2/p'} .
\end{equation}
We now want to get rid of the dependence in $t$ of the limits of the second integral. To do this, if we use the bound $e^{-Ms} \le 1$, we observe that by Hölder inequality
 \begin{equation*}
    \int_0^\varepsilon  \int_0^t  \left| Q_1(s)-Q_2(s) \right|e^{-Ms}  \,ds \, dt  \le  2\max_{j= 1,2}\norm{Q_j}_{L^p(0,a)}\int_0^\varepsilon   t^{1/p'}  \, dt \le  M\varepsilon^{1+1/p'} ,
\end{equation*}
and, analogously,
 \begin{equation*} 
    \int_0^\varepsilon  \int_{a+t-\varepsilon}^a  \left| Q_1(s)-Q_2(s) \right|e^{-Ms} \,ds \, dt  \le  2\max_{j= 1,2}\norm{Q_j}_{L^p(0,a)}\int_0^\varepsilon  (\varepsilon-t)^{1/p'}  \, dt \le  M\varepsilon^{1+1/p'} .
\end{equation*}
Combining these two observations we obtain that
\begin{align*}
  \int_0^\varepsilon  \int_t^{a+t-\varepsilon}  \left| Q_1(s)-Q_2(s) \right|& e^{-Ms}  \,ds \, dt \\
  &\ge  
    \int_0^\varepsilon  \int_0^{a}  \left| Q_1(s)-Q_2(s) \right|e^{-Ms}  \,ds \, dt -  2M\varepsilon^{1+1/p'} \\
    &= \varepsilon   \int_0^{a}  \left| Q_1(s)-Q_2(s) \right| e^{-Ms}  \, ds  - 2M\varepsilon^{1+1/p'}.
\end{align*}
Inserting this estimate in \eqref{est:stab_final_1} and using that $\varepsilon^{2/p'}<\varepsilon^{1/p'}$ (recall $\varepsilon<1$), gives
\begin{equation}\label{e:main_stab_estimate}
  \int_0^{a}  \left| Q_1(s)-Q_2(s) \right| e^{-Ms} \,ds    < 
          e^{2Ma  } g_a(0)\varepsilon^{-1}     
         + 4M\varepsilon^{1/p'}.
\end{equation}
 We first prove the local estimate. Since $a$ is fixed, we optimize the estimate by choosing
 $\varepsilon  =  g_a(0)^{\frac{p'}{p'+1}}  $, so it follows that
\begin{equation*}
     \int_0^{a}  \left| Q_1(s)-Q_2(s) \right|e^{-Ms}  \,ds    <
  \left(e^{2Ma  }  + 4M  \right) g_a(0)^{\frac{1}{p'+1}} ,
\end{equation*}
which, together with \eqref{id:def_g_a} and that $e^{-aM} \le e^{-Mt} \le 1$ in $[0,a]$ proves the local Hölder estimate. Since $\varepsilon<\min(1,a)=\varepsilon_0$ we require \eqref{est:A_condition} to hold (notice also that \Cref{thm:simon_uniqueness} implies  $g_a(0) \neq 0$ and---and hence $\varepsilon \neq 0$---if $Q_1\neq Q_2$ on $(0,a)$).

It remains to prove the global estimate. We now assume that \eqref{id:def_Q_loc_bound} holds in all $\R_+$:
 \begin{equation} \label{e:norm_global}
\max_{j=1,2}  \left(\norm{Q_j}_{L^1(\R_+)} + \norm{Q_j}_{L^p(\R_+)}  \right) \le    2^{-1}M .
\end{equation}
First, observe that
\begin{equation*}
    \int_0^{\infty}  \left| Q_1(s)-Q_2(s) \right|e^{-Ms}  \,ds \le \int_0^{a}  \left| Q_1(s)-Q_2(s) \right|e^{-Ms}  \,ds +Me^{-Ma}
\end{equation*}
Inserting this in \eqref{e:main_stab_estimate} and using that $g_a(0) \le g_\infty(0)$ yields
\begin{equation*}
  \int_0^\infty  \left| Q_1(s)-Q_2(s) \right| e^{-Ms} \,ds    < 
         e^{2Ma  }  g_\infty(0)\varepsilon^{-1}    
         +4M\varepsilon^{1/p'} +Me^{-Ma}.
\end{equation*}
This estimate holds true for any $1<a<\infty$ and $\varepsilon<1$. In particular, we can choose $a = 1 - \log\varepsilon $. This gives
\begin{equation*}
  \int_0^\infty  \left| Q_1(s)-Q_2(s) \right| e^{-Ms} \,ds    < 
        e^{2M}  g_\infty(0)\varepsilon^{-1 - 2M}  +4M\varepsilon^{1/p'} +M e^{-M} \varepsilon^{M},
\end{equation*}
for all $\varepsilon<1$. 
Finally, we put $\varepsilon = g_\infty(0)^\frac{1}{\la}$ with $\la>0$ and select the optimal value $\la = 1 + 2M + 1/p'$ to obtain
\begin{equation*}
  \int_0^\infty  \left| Q_1(s)-Q_2(s) \right| e^{-Ms} \,ds    < 
         (e^{2M}+4M)   g_\infty(0)^\beta  
          +Me^{-M}g_\infty(0)^{Mp' \beta},
\end{equation*}
where $\beta =(p'\lambda)^{-1} =  (p'(1+2M)+1)^{-1}$. The previous choice of $\varepsilon$ implies we need to assume $g_\infty(0)<1$ since $\varepsilon<1$. Then, using that $g_\infty(0)^{Mp' \beta} <g_\infty(0)^\beta$ since $Mp'>1$, leads to 
\begin{equation*}
  \int_0^\infty  \left| Q_1(s)-Q_2(s) \right| e^{-Ms} \,ds    < 
          ( e^{2M} +5M) g_\infty(0)^\beta ,
\end{equation*}
since $e^{-M}<1$. This finishes the proof of the theorem, since
\begin{equation*}
    g_\infty(0) =\int_0^\infty  \left| A_1(t) - A_2(t) \right|e^{-Mt} \, dt.
\end{equation*}
It is important to notice that, by \Cref{l:Da} i), the estimate obtained is not trivial, since $ g_\infty(0)< \infty$ always if  \eqref{e:norm_global} holds.
\end{proof}


\subsection{Stability of the Born approximation}

The analysis of the problem on the half line yields a Hölder stability estimate for the Born approximation. The proof of the local stability result follows directly from \Cref{thm:stability_A_function}.

\begin{proof}[Proof of \Cref{mt:stability_local}]
     Define $Q_j$ and $A_j$ for $j=1,2$ by
    \begin{equation} \label{id:V_VB_transform_j}
    V_j(x)  = |x|^{-2}Q_j(-\log|x|), \quad \text{and} \quad  \ve_j(x)  = |x|^{-2}A_j(-\log|x|).
    \end{equation}
    Write $a = -\log b$. Using the change of variables
    \eqref{id:V_VB_transform_j}, we have for all $1 \le p <\infty$
    \begin{equation*}
        \norm{Q_j}_{L^p(0,a)} = \frac{1}{|\IS^{d-1}|^{1/p}}\left(\int_{U_b}|V_j(x)|^p|x|^{2p-d}dx\right)^{1/p}
        \le \frac{\norm{V_j}_{L^p(U_b)}}{|\IS^{d-1}|^{1/p}} \max(b^{2p-d},1)^{1/p}.
    \end{equation*}
    Using this inequality  together with $\norm{Q_j}_{L^1(0,a)} \le  a^{1-\frac{1}{p}} \norm{Q_j}_{L^p(0,a)}$, we get
        \begin{equation*}
        \norm{Q_j}_{L^1(0,a)} +\norm{Q_j}_{L^p(0,a)} \le \norm{V_j}_{L^p(U_b)} \, \beta(d,b,p),
    \end{equation*}
    where $\beta(d,b,p) := \frac{ \max(b^{2p-d},1)^{1/p}}{|\IS^{d-1}|^{1/p}} (1+  (-\log b)^{1-\frac{1}{p}} )$.

    We now assume that $1<p<\infty$ in order to apply \Cref{thm:stability_A_function}. 
    
    First, if \eqref{e:pot_bound_local} holds for $N$, then \eqref{id:def_Q_loc_bound} holds for $M = 2\beta(d,b,p) N$. One can verify that  $M>1$, as required in \Cref{thm:stability_A_function}, using the assumption  $N > |\IS^{d-1}|^{1/p}$ together with the lower bound $\beta(d,b,p)> |\IS^{d-1}|^{-1/p}$. 
    Now, using again  \eqref{id:V_VB_transform} we have
    \begin{multline*}
         \int_0^a  \left| A_1(t) - A_2(t) \right|e^{-Mt}  \, dt  \le   |\IS^{d-1}|^{-1}\int_{U_b} |\ve_1(x)-\ve_2(x)||x|^{M-(d-2)} \, dx  \\
         \le  |\IS^{d-1}|^{-1} \max \left(1,b^{M-(d-2)} \right) \norm{\ve_1-\ve_2}_{L^1(U_b)},
    \end{multline*}
    so if $\delta_{d,p,b} := \max(1,b^{M-(d-2)})^{-1} |\IS^{d-1}| \varepsilon_0^{(1+p')/p'}$ with $1/p' = 1-1/p$, then \eqref{e:pot_bound_local} implies that \eqref{est:A_condition} holds. Thus, the conditions to apply \Cref{thm:stability_A_function} are met, so using the change of variables \eqref{id:V_VB_transform} in \eqref{e:stab_A_loc} we get 
    \begin{equation*}
       \int_{U_b}|V_1(x)-V_2(x)||x|^{M-(d-2)} \, dx \\ <C_{d,N,p,b} \left(\int_{U_b}|\ve_1(x)-\ve_2(x)||x|^{M-(d-2)}\, dx\right)^{\frac{p}{2p-1}}.
    \end{equation*}
    To finish the proof, one removes the weights using the fact that they are bounded below and above in $U_b$. The case $p=\infty$ is proved analogously.
\end{proof}
The proof of \Cref{mt:stability_global} will require two technical results.
\begin{lemma} \label{lemma:byebye_weight}
    Let $d\ge 2$, $1<p<\infty$, $\mu>0$, and $N>0$. There exists a constant $C_{d,p,N,\mu}>0$ such that, if
    \begin{equation}\label{e:N_g_bound}
        \norm{g}_{L^p(\IB^d)} \le N,
    \end{equation}
    then
    \begin{equation*}
        \norm{g}_{L^1(\IB^d)} \le C_{d,p,N,\mu} \left(  \int_{\IB^d} |g(x)| |x|^{\mu}   \, dx \right)^{\alpha} \quad \text{with} \quad \alpha: = \tfrac{d }{\left(d + \mu\frac{p}{p-1} \right)}.
    \end{equation*}
\end{lemma}
\begin{proof}
    Let $0<\rho \le 1$ and  $I:=\int_{\IB^d} |g(x)| |x|^{\mu} \, dx$. Using Hölder's inequality and \eqref{e:N_g_bound} it follows that
    \begin{equation*}
         \norm{g}_{L^1(\IB^d)} \le |\IB^d_\rho|^{1/p'} \norm{g}_{L^p(\IB^d)} + \rho^{-\mu}\int_{U_\rho} |g(x)| |x|^{\mu} \, dx  \le  A N \rho^{\mu_0}  + \rho^{-\mu} I  ,
    \end{equation*}
    where $\mu_0:= d/p'$, $A: = (|\IS^{d-1}|/d)+^{1/p'}$, and $1/p' = 1-1/p$.
    The quantity in the RHS, as a function of $\rho$, has a unique critical point in $(0,\infty)$, namely
    \begin{equation} \label{e:rho_zero}
        \rho_0:= \left(\frac{\mu  I}{ \mu_0 AN}\right)^{\beta},\qquad \beta := \frac{1}{\mu+\mu_0},  
    \end{equation}
    which is a global minimum. Since we have the restriction $\rho\le 1$, we consider two cases.
    First, assume  $\rho_0<1$. Then inserting \eqref{e:rho_zero} in the previous estimate and taking into account that $\beta\mu_0=1-\beta\mu$, one gets
    \begin{equation*}
         \norm{g}_{L^1(\IB^d)}  \le A N \left(\frac{\mu  I}{ \mu_0 AN}\right)^{\beta\mu_0}  +  I  \left(\frac{\mu  I}{ \mu_0 AN}\right)^{-\beta \mu} \le C_{d,p,N,\mu} I^{\alpha},
    \end{equation*}
    where $\alpha=\beta\mu_0=1-\beta\mu =\frac{d }{p'\mu+ d }$.
    
    Otherwise, assume $\rho_0 \ge 1$. On the one hand, \eqref{e:rho_zero} and $\rho_0 \ge 1$ imply $AN < I \frac{\mu}{\mu_0}$. On the other hand, applying Hölder and \eqref{e:N_g_bound} one has $\norm{g}_{L^1(\IB^d)}  \le A N$. Thus one gets
    \begin{equation*}
          \norm{g}_{L^1(\IB^d)}  \le  A N  =  (AN)^{1-\alpha} (AN)^{\alpha} \le (AN)^{1-\alpha}\left(\frac{\mu}{\mu_0}\right)^{\alpha}  I^{\alpha}.
    \end{equation*}
    This finishes the proof of the lemma.
\end{proof}

\begin{lemma} \label{lemma:finitness_V}
    Let $V\in L^p_\rad(\IB^d,\R)$ for some $d/2<p<\infty$. For every $M>0$ such that
    \begin{equation}\label{e:hypV}
         M \ge \beta(d,p)  \norm{V}_{L^p(\IB^d)} , \qquad \beta(d,p) : =  \frac{2}{|\IS^{d-1}|^{1/p}}\frac{p-1}{p-d/2}
    \end{equation}
    one has $\displaystyle{\int_{\IB^d}|\ve(x)| |x|^{M-(d-2)} \,dx \le |\IS^{d-1}|\beta(d,p) \norm{V}_{L^p(\IB^d)}}$.
\end{lemma}
\begin{proof}
    Let $Q$ be defined from $V$ via \eqref{id:V_VB_transform}. Then, by Hölder's inequality,
    \begin{equation} \label{e:Q_j_L1}
        \norm{Q}_{L^1(\R_+)} = \frac{1}{|\IS^{d-1}|} \int_{\IB^d}|V(x)||x|^{2-d}dx 
        \le \frac{1}{2} \beta(d,p)\norm{V}_{L^p(\IB^d)},
    \end{equation}
    since
     \[
    \left(\int_{\IB^d}|x|^{p'(2-d)}dx\right)^{1/p'}\leq |\IS^{d-1}|^{1/p'}\frac{p-1}{p-d/2}.
    \]
    Assumption \eqref{e:hypV} and \eqref{e:Q_j_L1} imply that  $\norm{Q}_{L^1(\R_+)}<M/2$. By \Cref{l:Da} and \eqref{id:V_VB_transform}, one has that 
    \begin{multline*}
        \int_{\IB^d} |\ve(x)| |x|^{M+2-d}\,dx = \int_{\IB^d} \frac{|A(-\log|x|)|}{|x|^{d-M}}\,dx
        \\
        =|\IS^{d-1}|  \int_0^\infty |A(t)|e^{-Mt}\,dt <2|\IS^{d-1}| \norm{Q}_{L^1(\R_+)}.
    \end{multline*}
    Using again \eqref{e:Q_j_L1} finishes the proof of the lemma.
\end{proof}
\begin{proof}[Proof of \Cref{mt:stability_global}]
    The proof follows from \Cref{thm:stability_A_function} when $a =\infty$, as we now show.  Define $Q_j$ and $A_j$ for $j=1,2$ by \eqref{id:V_VB_transform_j}.
    We start by observing that 
    \begin{equation*}
        \norm{Q_j}_{L^p(\R_+)} = \frac{1}{|\IS^{d-1}|^{1/p}}\left(\int_{\IB^d}|V_j(x)|^p|x|^{2p-d}dx\right)^{1/p}
        \le \frac{\norm{V_j}_{L^p(\IB^d)}}{|\IS^{d-1}|^{1/p}},
    \end{equation*}
    since $p>d/2$. This estimate and \eqref{e:Q_j_L1} show that \eqref{id:def_Q_loc_bound}  holds for $a = \infty$ whenever \eqref{e:def_pot_bound} is satisfied. 
    In addition, \eqref{e:def_pot_bound} implies that \eqref{est:A_condition} holds with $a=\infty$ (in this case $\varepsilon_0 =1$). Since $M>d-1$, in particular $M>1$ for $d\ge 2$.
    Thus, the conditions to apply \Cref{thm:stability_A_function} are met, which implies that \eqref{is:Asta_global} holds. 
    This yields, using  \eqref{id:V_VB_transform},
    \begin{equation*}
       \int_{\IB^d}|V_1(x)-V_2(x)||x|^{M-(d-2)} \, dx \\ <C_{d,M,p}\left(\int_{\IB^d}|\ve_1(x)-\ve_2(x)||x|^{M-(d-2)}\, dx\right)^\beta,
    \end{equation*}
    where $C_{d,M,p}=|\IS^{d-1}|^{1-\beta}(e^{2M}+5M)$ and $\beta = (1+\frac{p}{p-1}(1+2M))^{-1}$. We cannot remove the weight on the RHS, since $\ve_j$ might be very singular at the origin, but since $M>d-1>d-2$, we can remove the weight $|x|^{M-(d-2)}$  on the LHS using \Cref{lemma:byebye_weight} together with \eqref{e:def_pot_bound}. This  proves \eqref{e:glob_sta} with 
    \[
     \frac{d}{\left(d+ 3p(p-1)^{-1}M\right)^2} <\alpha<1\,.\qedhere
   \] 
\end{proof}


\section{Canonical regularization of the Born approximation} \label{sec:regularization}
 
Here we denote by $\cE'(U)$ the space of compactly supported distributions on an open set $U \subseteq \R^d$. For simplicity, in this section we will consider $\cE'(U)\subseteq \cE'(\R^d)$ with the natural embedding provided by the extension by zero outside $U$. The Paley-Wiener theorem ensures that given any $f\in\cE'(\R^d)$, its Fourier transform:
\begin{equation}\label{id:fourier_compact_supp}
    \widehat{f}(\xi):=\langle f, e_{-i\xi}\rangle_{\cE'\times\cC^\infty},\qquad e_{\xi}(x):=e^{\xi\cdot x},\qquad x,\xi\in\R^d, 
\end{equation}
extends to an entire function in $\IC^d$. 

The moments $\sigma_k[f]$ of a distribution $f\in \cE'(\R^d)$  are defined by 
\begin{equation}  \label{id:distribution_moments}
    \sigma_k[f] := |\IS^{d-1}|^{-1} \br{f,m_k}_{\cE'\times\cC^\infty},\quad m_k(x):=|x|^{2k}, \qquad \forall k\in \N_0.
\end{equation}
A distribution $f\in \cE'(\R^d)$ is \textit{radially symmetric} (or just radial) if and only if 
\[
\br{f, \varphi\circ \rho}_{\cD'\times\cC^\infty_c} =  \br{f,   \varphi}_{\cD'\times\cC^\infty_c},  \qquad  \forall \rho\in \SO(d),\quad \forall \varphi\in \cC^\infty_c(\R^d).
\]

\Cref{mt:existence}(ii) follows directly from  the following result.
\begin{theorem} \label{mt:regularization}
Let $d\ge 2$ and $V \in \cV_d$. There exists a unique radial distribution $V^\mB_r \in \cE'(\IB^d) + L_\rad^1(\IB^d)$ such that
        \begin{equation} \label{id:moment_regularization_sec7}
            \sigma_k[V^\mB_r] = \lambda_k[V]-k,  \qquad \forall k\in\N_0. 
        \end{equation}
In addition, $V^\mB_r$ is a regularization of $\ve$, namely,
\[
\br{V^\mB_r,\varphi}_{\cE'\times\cC^\infty}=\int_{\IB^d\setminus\{0\} } \ve(x) \varphi(x)  \, dx ,\qquad \forall \varphi\in\cC^\infty_c(\IB^d\setminus\{0\}),
\]
and the Fourier transform of $V^\mB_r$ satisfies the  following identities:
\begin{enumerate}[i)]
    \item  For all $\xi \in \R^d$
        \begin{equation}  \label{id:regularization}
         \widehat{\ve_r}(\xi)
        =  2 \pi^{d/2}  \sum_{k=0}^\infty  \frac{ (-1)^k}{k! \Gamma(k+d/2)}
        \left(\frac{|\xi|}{2}\right)^{2k}   (\lambda_k[V] - k).
        \end{equation}
    \item   If $\Lambda_V$ is well defined then, for every $\xi \in \R^d\setminus \{0\}$, and  for all $\zeta_1,\zeta_2\in \IC^d$ such that $\zeta_1\cdot \zeta_1 = \zeta_2\cdot \zeta_2 =0$  and $\zeta_1+\zeta_2 = -i\xi$ the following holds
        \begin{equation}  \label{id:syl_Uhl}
        \widehat{\ve_r}(\xi) =   (\overline{e_{\zeta_1}},(\Lambda_V -\Lambda_0)e_{\zeta_2})_{L^2(\IS^{d-1})},
    \end{equation}
    where, for $\zeta\in\IC^d$, we have written $e_\zeta(x):=e^{\zeta\cdot x}$.
\end{enumerate}
\end{theorem}

Note that the fact that  $(\lambda_k[V]-k)_{k\in \N_0}$ is the sequence of moments of a radial distribution is a non-trivial information on the structure of radial DtN maps (see \Cref{sec:appendix_moments}). Since $\ve_r$ coincides exactly with $\ve$ outside the origin, formula \eqref{id:regularization} offers an explicit method to reconstruct $\ve_r$ and, therefore $\ve$, from  $(\lambda_k[V]-k)_{k\in \N_0}$.
Identity \eqref{id:syl_Uhl} connects the concept of the Born approximation with the method of Complex Geometrical Optics solutions of the Schrödinger equation of \cite{SU87} and was used in \cite{BCMM22} to introduce the Born approximation in the context of the Calderón problem.

To prove \Cref{mt:regularization} we need to show that the Fourier transform of a radial compactly supported distribution can always be reconstructed from the moments \eqref{id:distribution_moments} by an explicit formula. The following extends \cite[Identity~(1.20)]{BCMM22} to  distributions.
\begin{lemma} \label{lemma:fourier_transform_identity}
    Let $f \in \cE'(\R^d)$ be radially symmetric. Then
    \begin{equation}\label{id:ft}
        \widehat{f}(\xi)
        =  2 \pi^{d/2}  \sum_{k=0}^\infty  \frac{ (-1)^k}{k! \Gamma(k+d/2)}
        \left(\frac{|\xi|}{2}\right)^{2k}   \sigma_k[f] .
    \end{equation}
\end{lemma}
\begin{proof}
If $f \in \cE'(\R^d)$  then the Paley-Wiener theorem ensures that $\widehat {f}$ is an entire function on $\R^d$ (see, for example, \cite[Theorem 7.3.1]{HormanderI}). 
Moreover $f$ is radial if and only if $\widehat {f}$ is   radial.  Therefore it must hold that
\begin{equation*}
    \widehat{f}(\xi) = \sum_{k=0}^\infty  a_k  |\xi|^{2k}, 
\end{equation*}
for some appropriate coefficients $ a_k \in \IC$.

On the other hand,  $a_k  =  b_k (-\Delta)^k \widehat{f}(0)$ for all $k\in\N_0$, where $(b_k)_{k\in\N_0}$ are  some coefficients independent  of $f$---notice that $b_k$  is essentially a coefficient of the Taylor expansion of the radial profile function of $\widehat{f}$. Using  \eqref{id:fourier_compact_supp} one can show
that
\[
(-\Delta)^k \widehat{f}(0) =  \br{f,m_k}_{\cE'\times\cC^\infty} = |\IS^{d-1}|\sigma_k[f].
\]
Hence we conclude that
\begin{equation} \label{id:b_k_sigma}
    \widehat{f}(\xi) = \sum_{k=0}^\infty b_k  |\IS^{d-1}| |\xi|^{2k}  \sigma_k [f] ,
\end{equation}
where the $(b_k)_{k\in\N_0}$ coefficients are  independent of $f\in\cE'(\R^d)$. On the other hand, formula \eqref{id:ft} is proved in \cite[p.~19]{BCMM22} for compactly supported $f\in L^1(\R^d)$. Using this, and the fact that $(b_k)_{k\in\N_0}$ are universal, concludes the proof of the lemma.
\end{proof}

In what follows, given $V\in\cV_d$, we will denote by $\ve_e$ the extension by zero    of $\ve$ to $\R^d$.
Since \Cref{mt:existence}(i) ensures that $\ve \in L^{1}_{\loc} (\IB^d\setminus \{0\})$,  one automatically has $\ve_e \in L^{1}_{\loc}(\R^d\setminus \{0\})$. 
\begin{proof}[Proof of $\Cref{mt:regularization}$]  
By \Cref{mt:approximation} we know that   $\ve_e \in L^1_{\loc}(\R^d \setminus \{0\}) $ and that at $x=0$ it has, at worst, a singularity of order $|x|^{-\alpha}$ for some $\alpha>0$ that depends on $V$. Then, by \cite[Proposition 1 p. 11]{GelShilVol1}  there is always an extension $F\in \cD'(\R^d)$ of $\ve_e$ such that 
    \begin{equation*}
        \br{F,\varphi}_{\cD'\times\cC^\infty_c} =   \int_{\R^d\setminus\{0\} } \ve_e(x) \varphi(x)  \, dx \qquad \forall \varphi\in \cC^\infty_c(\R^d\setminus\{0\}).
    \end{equation*}
    Such an extension is called a regularization of the singular function $\ve_e$. In particular, we have that $F\in \cE'(\R^d)$ since it coincides with  $\ve_e$ outside the origin, and hence vanishes outside $\ol{\IB^d}$. Also, since $\ve_e$ is radial, one can always choose $F$ to be radial. Notice that two different regularizations of $\ve_e$ differ in a distribution supported at $x=0$, or in other words, in a finite linear combination of derivatives of the Dirac delta distribution $\delta_0$ which is supported at $x=0$.

    We now claim that, since there exists an $N\in \N_0$ large enough such that $m_N\ve_e \in L^1(\R^d)$,  for every regularization $F$ of $\ve_e$ one can always find an $N' \in \N_0$ such that
      \begin{equation*}
        \br{m_k  F ,\varphi}_{\cE'\times\cC^\infty} =  \br{m_k\ve_e,\varphi}_{\cE'\times\cC^\infty} \quad \text{ for all } \varphi\in \cC^\infty_c(\R^d), \text{ and for all } k\ge N'.
    \end{equation*}
    In other words, $m_kF  = m_k\ve_e$ as distributions in $\cE'(\R^d)$, for all $k\ge N'$.

    As an immediate consequence of this, we obtain that
    \begin{equation} \label{id:moment_regularized}
        \sigma_k[F] = \sigma_k[\ve_e] = \lambda_k[V]-k \quad \text{for all }  k \ge N'.
    \end{equation}

    To prove the claim, start by observing that $m_N F$ and  $m_N\ve_e$ are both compactly supported distributions that are identical outside $x=0$, so they differ only in a  finite linear combination of derivatives of $\delta_0$. Let $M$ be the maximum order of the derivatives of $\delta_0$.
    Therefore, for any  $N'\ge N$ large enough
    \begin{align*}
        m_{N'}F - m_{N'}\ve_e
        &= m_{N'-N}(m_{N}F - m_{N}\ve_e) \\
        &= m_{N'-N} \sum_{|\alpha| \le M} c_\alpha \partial^\alpha_x \delta_0 = 0,
    \end{align*}
where the right hand side will vanish provided  $N'-N > M$.

Now, define formally $V^\mB_r$ by
\begin{equation*}
     \widehat{\ve_r}(\xi)
        :=  2 \pi^{d/2}  \sum_{k=0}^\infty  \frac{ (-1)^k}{k! \Gamma(k+d/2)}
        \left(\frac{|\xi|}{2}\right)^{2k}   (\lambda_k[V] - k) .
\end{equation*}
An immediate consequence of \Cref{lemma:fourier_transform_identity} and \eqref{id:moment_regularized} is that
\begin{equation*}
\widehat{F}(\xi) =  \widehat{\ve_r}(\xi) + P(|\xi|^2),
\end{equation*}
where $P$ is a polynomial of  order $N'$ at most. This proves that $\widehat{\ve_r} \in \cS'(\R^d) $  where  $\cS'(\R^d)$ denotes the set of tempered distributions, \textit{i.e.} the dual of the Schwartz class. 
Since the Fourier transform is an isomorphism on $\cS'(\R^d)$, it follows that $V^\mB_r \in \cS'(\R^d)$ is well defined. 
Moreover, since the inverse Fourier transform of $P(|\xi|^2)$ is a linear combination of derivatives of $\delta_0$, we actually have that $V^\mB_r \in \cE'(\R^d)$ with support $\ol{\IB^d}$, and that $V^\mB_r = \ve$ in the distribution sense on $\IB^d\setminus\{0\}$, which proves that $V_r^\mB$ is a regularization of $\ve$ in $\IB^d$. 

This implies that $V_r^\mB\in \cE'(\IB^d)+L^1_\rad(\IB^d)$, since $V_r^\mB$ coincides with an integrable function in a neighborhood of the boundary (here we abuse notation by identifying $\ve_r \in \cE'(\R^d)$ with its restriction to $\IB^d$, which is well defined since $\ve_r$ is integrable close to the boundary).  

From the formula defining $\widehat{V_r^\mB}(\xi)$ we also get that $ V_r^\mB$ is a radial distribution. In addition, \eqref{id:moment_regularization_sec7} follows from the definition of $\ve_r$ \Cref{lemma:fourier_transform_identity} and that $\ve_r$ is uniquely determined by this condition.
The fact that  $V_r^\mB$ satisfies the identity (ii) of the statement is a consequence of \cite[Theorem 1]{BCMM22}.
\end{proof}


\section{A partial characterization of DtN operators} \label{sec:appendix_moments}  

In this section, we show that the existence of the Born approximation has strong consequences for the structure of the spectrum of a radial DtN map. 

Let $V \in \cV_d$, and denote, as usual, by $(\lambda_k[V])_{k\in\N_0}$ the eigenvalues of the DtN map $\Lambda_V$, as defined in \Cref{sec:subsec_DtNmap}. We set:
\begin{equation*}
    \upsilon_n[V]:= \lambda_{n}[V]- n,  \qquad n\in\N_0, 
\end{equation*}
and write, for any sequence of complex numbers $(\mu_n)_{n\in\N_0}$, 
\begin{equation*}
    \Delta^k \mu_n := \sum_{m=0}^k (-1)^m \binom{k}{m} \mu_{n+k-m}  ,\qquad k \in\N_0.
\end{equation*}
\begin{theorem}\label{thm:characterization}
    Suppose $V \in L^p_\rad (\IB^d;\R)$ with  $p > d/2$. Then the eigenvalues of $\Lambda_V$ satisfy, for a certain $k_0\geq 0$ that may depend only on $d$, $p$, and the $L^p$ norm of $V$:
    \begin{equation}\label{characterizationLp}
    \sup_{k\geq 0}\,(k+1)^{p-1}\sum_{m=0}^k \left|\binom{k}{m} \Delta^{k-m} \upsilon_{m+k_0}[V]\right|^p <\infty.
\end{equation}
\end{theorem}
\begin{proof}
    We recall that the Born approximation of $V\in\cV_d$ satisfies $\ve (x) \in L^1 (\IB^d, |x|^{2k_0} dx)$ with $k_0 = \lfloor k_V\rfloor +1>\lfloor k_Q\rfloor +1$ and that is given by 
    $\ve (x)  = |x|^{-2}A(-\log |x|)$ (see (\ref{id:VB_def})).

    Let $q^\mB (r):=r^{-2}A(-\log r)$. It follows that for $k \geq k_0$,
    \begin{equation*}
    \sigma_k [\ve] = \int_0^1 q^\mB (r) r^{2k+d-1} \ dr = \int_0^1 \left( \frac{1}{2}  q^\mB (\sqrt t) \ t^{\frac{d}{2} -1} \right) \ t^k \ dt. 
    \end{equation*}
    Clearly, 
    \begin{equation} \label{e:Qb}
    Q^\mB(t) := \frac{1}{2}  q^\mB (\sqrt t) \ t^{\frac{d}{2} -1 +k_0},
    \end{equation}
    verifies $Q^\mB\in L^1 ((0,1))$. 
    For $f \in L^1 ((0,1))$, we define the Hausdorff moments by 
    \begin{equation*}
    \mu_k[f] = \int_{0}^1 t^{k} f(t) \ dt , \qquad \forall k\in\N_0.
    \end{equation*}
    so that one has, by \Cref{mt:existence}, 
    \begin{equation}\label{id:moment_vb}
    \mu_k [Q^\mB]=\sigma_{k+k_0}[\ve]=\upsilon_{k+k_0}[V],\qquad \forall k\in\N_0.
    \end{equation}

    Since the classical Hausdorff moment problem possesses a unique solution in $L^1((0,1))$ (see \cite[Chapter III]{widd41}), it follows that $Q^\mB$ is the {\it{unique}} function in $ L^1 ((0,1))$ such that 
    $\mu_n [Q^\mB] = \upsilon_{n+k_0}[V]$ for all $n\in\N_0$. Sequences of moments of functions in $L^1((0,1))$ have been characterized by Hausdorff; however, the implications on the spectrum of the DtN map on a potential in $V \in \cV_d$ are not particularly simple to state (see \Cref{remark:charVd}).  

    Suppose now $V \in L^p_\rad (\IB^d;\R)$ with  $p > d/2$. For such values of $p$, the function $F(r)$ introduced in \Cref{mt:approximation} satisfies $q^\mB(r)= q(r)+F(r)$ and the following estimate:
    \begin{equation*}
    | F(r)| \leq \frac {C}{r^{\beta+2}},\qquad 0<r\leq 1,
    \end{equation*}
    where $C, \beta$ are suitable constants that depend only on $d$, $p$, and $\norm{V}_{L^p(\IB^d)} $.
    Thus $q^\mB(r)$ satisfies
    \begin{equation*}
        \int_0^1 | q^\mB (r) |^p \ r^{d-1+2k_0} \ dr < \infty,
    \end{equation*}
    where now we take $k_0 = \max\left(\big\lfloor \frac{p(\beta+2)-d}{2} \big\rfloor +1,\lfloor k_V\rfloor +1\right)$.
    It follows that $ Q^\mB$ defined in \eqref{e:Qb} and satisfying \eqref{id:moment_vb} is a function in $ L^p ((0,1))$. 
    Now, using \cite[Theorem~5, p. 110]{widd41}, we see that $(\mu_n[Q^\mB])_{n\in \N_0}$ are  the Hausdorff moments of $Q^\mB \in L^p ((0,1))$ if and only if  
    \begin{equation*}
    \sup_{k\in\N_0}(k+1)^{p-1}\sum_{m=0}^k | \nu_{k,m}[Q^\mB]|^p <\infty,
    \end{equation*}
    where, following \cite[p. 101]{widd41}, we write 
    \begin{equation*}
    \nu_{k,m}[Q^\mB] := \binom{k}{m} (-1)^{k-m} \Delta^{k-m} \mu_m[Q^\mB]  ,\qquad k \geq m \geq 0.
    \end{equation*} 
    To conclude, note that, since the difference operators $\Delta^k$ commute with the right shift, one has the relations:
    \begin{equation*}
        \nu_{k,m}[Q^\mB]=\binom{k}{m}(-1)^{k-m}\Delta^{k-m} \upsilon_{m+k_0}[V], \qquad k \geq m \geq 0. \qedhere
    \end{equation*}
\end{proof}    
\begin{remark}\label{remark:charVd}
    The proof of \Cref{thm:characterization} also gives information on the eigenvalues when one merely has $V \in \cV_d$. For $k \geq 1$, define 
    \begin{equation*}
    L_k (t) := (k+1) \nu_{k, \lfloor kt \rfloor}[Q^\mB] ,\  \ t \in [0,1]. 
    \end{equation*}
    It is shown in \cite[p. 112]{widd41} that $(\mu_n)_{n\in \N_0}$ are  the Hausdorff moments of $Q^\mB \in L^1 ((0,1))$ if and only if the sequence $(L_n (t))_{n\in\N}$ converges in $L^1 ((0,1))$. This is not a total characterization of the DtN operators, see \Cref{r:last}, and  this result is not easily interpreted, independently of the functions $L_n (t)$, in terms of the sequence $(\upsilon_n[V])_{n\in\N_0}$.
\end{remark}

\begin{remark} \label{r:last}
Condition (\ref{characterizationLp}) can be viewed as a partial characterization of DtN operators for radial potentials $V \in L^p_\rad(\IB^d;\R)$, $p > d/2$. Nonetheless, a total characterization should involve additional conditions. This is due to the fact that not every locally integrable function is the $A$-amplitude of a Schrödinger operator on the half-line, as has been shown by Remling \cite{Remling03}.
\end{remark}


\appendix
\section{Solutions by separation of variables} 
\label{sec:appendix_lemma}

\begin{lemma} \label{lemma:radial_solutions_appendix}
   Let $d\ge 2$ and $V \in L^p(\IB^d)$ with $p>1$ and $p\ge d/2$ be a radial function. Then, for every $k> k_V$  there is a unique solution $b_k$  of \eqref{id:radial_b} with $b_k(1) = 1$ such that the function $u_k(x) = b_k(|x|) Y_k(x/|x|)$  is a proper weak  $ H^1(\IB^d)$ solution of \eqref{id:calderon_direct_Yk}. Moreover, for any  other solution $u \in H^1(\IB^d)$ of \eqref{id:calderon_direct_Yk}   it holds that
\begin{equation*}
    \partial_r b_k (1) = (Y_k, \partial_\nu u|_{\IS^{d-1}} )_{L^2(\IS^{d-1})},
\end{equation*}
for all $Y_k \in \gH_k$ with $\norm{Y_k}_{L^2(\IS^{d-1})} =1$.
\end{lemma}
\begin{proof} 
Let $u_0$ be an $H^1_0(\IB^d)$ solution of 
\begin{equation}  \label{id:calderon_homogeneous}
\left\{
\begin{array}{rlr}
-\Delta u_0 +V u_0 =& 0  &\text{ on }   \IB^d,\\
u_0 |_{  \IS^{d-1}} =& 0, &    \\
\end{array}\right.
\end{equation}

Since $V$ is radial, using a Fourier expansion in spherical harmonics of $u_0$, one can show that  $u_0(x) = \sum_{k=0}^\infty a_k(|x|) Y_k(x/|x|)$, for some $Y_k \in \gH_k$, 
where
$a_k(r)$ is a solution of \eqref{id:radial_b} with $a_k(1) =0$. Defining $v_k(t) := e^{-\frac{d-2}{2}t} a_k(e^{-t})$ and $Q$ from the radial profile of $V$ by \eqref{id:q_to_Q}, we obtain that $v_k$ solves
\begin{equation*}  
\left\{
\begin{array}{rlr}
-v_k''  + Q    v_k =& -\left(k+ \frac{d-2}{2}\right)^2 v_k  &\text{ on }   \R_+,\\
v_k(0) =& 0. &    \\
\end{array}\right.
\end{equation*} 
This implies that $v_k$ is a Dirichlet eigenfunction of the 1-$d$ operator \eqref{id:operator} with eigenvalue $-\left(k+ \frac{d-2}{2}\right)^2$. By \Cref{lemma:Schrodinger_exist}, we know that  $k+ \frac{d-2}{2} \le \beta_Q $, and therefore
\begin{equation} \label{is:u_fourier}
    u_0(x) = \sum_{0\le k \le k_V } a_k(|x|) Y_k(x/|x|),
\end{equation}
where we have used that $k_V = \beta_V - \frac{d-2}{2}$ and $\beta_Q \le \beta_V$, and imposed  $a_k =0$ whenever $k+ \frac{d-2}{2}>\beta_Q$.
Notice that any choice of $Y_k\in \gH_k$ in \eqref{is:u_fourier} gives a solution (not necessarily distinct) of the homogeneous problem \eqref{id:calderon_homogeneous}.

Problem \eqref{id:calderon_direct_ball} can be reduced to 
\begin{equation}  \label{id:calderon_homogeneous_f}
\left\{
\begin{array}{rlr}
-\Delta w +V w =& -V w_0   &\text{ on }   \IB^d,\\
w |_{  \IS^{d-1}} =& 0, &    \\
\end{array}\right.
\end{equation}
using the change of variables $u= w+w_0$, where $w_0$ is the unique harmonic function in $\IB^{d}$ satisfying $w_0|_{\IS^{d-1}} =f$ . Notice  that 
\begin{equation} \label{id:w_0_fourier}
w_0(x) = \sum_{k=0}^\infty |x|^{k} Y_k (x/|x|), \qquad Y_k = \Pi_{\gH_k} f,
\end{equation}
 where  $ \Pi_{\gH_k}$ stands for the $L^2(\IS^{d-1})$ projector to the subspace $\gH_k$ of spherical harmonics.
 
If $V\in L^{p}(\IB^d)$ with $p>1$ and $p\ge d/2$,  the standard theory of elliptic equations implies that \eqref{id:calderon_homogeneous_f} has a solution $w$ if and only if 
\[
(V w_0 ,u_0)_{L^2(\IB^{d})} = 0,
\]
for all $u_0$ that are solutions of the homogeneous problem \eqref{id:calderon_homogeneous},
see for example\footnote{ This is proved for bounded potentials, but the case $L^{d/2}(\IB^{d})$ can be proved using the same arguments.} \cite[Section 6.2.3]{evans}.
By the previous discussion we know that $u_0$ must satisfy \eqref{is:u_fourier}.
Therefore, using \eqref{id:w_0_fourier} and that $V$ is radial, one can verify that $(V w_0 ,u_0)_{L^2(\IB^{d})} = 0$ holds if we require $  \Pi_{\gH_k} f = 0$  for all $0\le  k\le k_V$.

Therefore, given $f$ such that 
\begin{equation*}
f = \sum_{k>k_V} Y_k,   \qquad Y_k \in \gH_k, 
\end{equation*}
there always exists a solution $w \in H^1(\IB^{d})$ of \eqref{id:calderon_homogeneous_f} and, as a consequence, a solution $u = w+w_0$ in $ H^1(\IB^{d})$ of \eqref{id:calderon_direct_ball}, even if \eqref{id:basic_assumption} does not necessarily hold. Using a Fourier expansion in spherical harmonics of $u$ in \eqref{id:calderon_direct_ball},  one can show that  $u$ must satisfy 
\[
u(x) = \sum_{k>k_V}^\infty b_k(|x|) Y_k(x/|x|) + u_0,
\]
 where $b_k$ solves \eqref{id:radial_b} with boundary conditions  $b_k(1) = 1$ if $k>k_V$, and $u_0$ is any homogeneous solution.  In particular it clearly holds by \eqref{is:u_fourier} that
 \begin{equation*}
    \partial_r b_k (1) = \frac{1}{\norm{Y_k}_{L^2(\IS^{d-1})}^2}(Y_k, \partial_\nu u|_{\IS^{d-1}} )_{L^2(\IS^{d-1})}.
\end{equation*}
 Also, since we can choose $u_0 = 0$  there is a unique solution $u_f \in H^1(\IB^d)$ of \eqref{id:calderon_direct_ball} such that
\[
u_f(x) = \sum_{k>k_V} b_k(|x|) Y_k(x/|x|).
\]
  In the particular case of $f= Y_k$ one obtains that  $u_k (x) : = u_f(x) =  b_k(|x|) Y_k(x/|x|)$. This is the only solution of separation of variables since any other solution differs in a homogeneous solution satisfying \eqref{is:u_fourier}.  This finishes the proof of the lemma.
\end{proof}

\section{The space \texorpdfstring{$\cV_d$}{Vd}}\label{sec:appendix_lorenz} 
 It is not difficult to show that the space $\cV_d$ defined in \eqref{id:def_Vd} contains $L_\rad^{d/2}(\IB^d)$. Note first that
\begin{equation*}  
\norm{V}_{L^{d/2}(\IB^d)} = |\IS^{d-1}|^{2/d} \, \norm{Q}_{L^{d/2}(\R_+)}.
\end{equation*}
Since
\begin{equation*}
 \int_{y}^{y+1}|Q(t)| \, dt  
 \leq  \left(\int_{y}^{y+1}|Q(t)|^{d/2} \, dt\right)^{2/d}
 \leq \,\norm{Q}_{L^{d/2}(\R_+)},
\end{equation*}
it follows that $\normm{Q}  \le \norm{Q}_{L^{d/2}(\R_+)}$, and as a consequence we obtain that 
\begin{equation}\label{est:Vinld2}
    \norm{V}_{\cV_d} \leq |\IS^{d-1}|^{(d-2)/d}    \norm{V}_{L^{d/2}(\IB^d)}.
\end{equation}
In fact, a stronger estimate holds for $d>2$.
\begin{lemma}  \label{lemma:lorentz}
Let $d>2$ and let $V \in L^{d/2, \infty}(\IB^d)$ be a (not necessarily radial \footnote{Notice that   $\norm{V}_{\cV_d}$ is well defined for non-radial potentials in \eqref{id:q_basic_assumption} even if, for convenience,  we have included the radial assumption in the definition of $\cV_d$.}) potential. Then
\begin{equation*}
   \norm{V}_{\cV_d} \leq C_d  \norm{V}_{L^{d/2, \infty}(\IB^d)}, 
\end{equation*}
where $C_d>0$ only depends of $d$.
\end{lemma}
\begin{proof} 
Let $A_j = \{x\in \R^d: 2^{-j-1} <|x|<2^{-j}\}$ and denote by $\chi_{A_j}$ the characteristic function of the set $A_j$.
    From \eqref{id:q_basic_assumption}, since $\frac{d}{d-2}$ is the Hölder conjugate exponent of $\frac{d}{2}$, it follows using Hölder inequality for Lorentz spaces  \cite{Hu1966} that
    \begin{align*}   
\norm{V}_{\cV_d} 
\le C_d\sup_{j\in \N_0} \norm{\chi_{A_j} |\cdot|^{2-d}}_{L^{\frac{d}{d-2},1}(\IB^d)} \norm{\chi_{A_j} V}_{L^{\frac{d}{2},\infty}(\IB^d)} \\
\le C_d \left( \sup_{j\in \N_0} \norm{\chi_{A_j} |\cdot|^{2-d}}_{L^{\frac{d}{d-2},1}(\IB^d)} \right) \norm{ V}_{L^{\frac{d}{2},\infty}(\IB^d)} .
\end{align*}  
To finish, we need to show that the factor with the $\sup_{j\in \N_0}$  is finite. The norm of the Lorentz space $L^{\frac{d}{d-2},1}(\IB^d)$ is given by
\[
  \norm{\chi_{A_j} |\cdot|^{2-d}}_{L^{\frac{d}{d-2},1}(\IB^d)} = \frac{d}{d-2}\int_0^\infty   |g_j(t)|^{\frac{d-2}{d}} \,dt,
\]
where $g_j(t)$ is the distribution function of $\chi_{A_j}|x|^{2-d}$, \textit{i. e.}
\[ 
g_j(t) = |\{x \in \IB^d:  2^{-j-1}<|x|<2^{-j} ,\, |x|^{2-d}>t   \} | .
\]
From an explicit computation of $g_j(t)$  it follows that
\begin{equation*}
    g_j(t) \le
    \frac{1}{d} \left(1-\frac{1}{2^d} \right)|\IS^{d-1}|  
    \begin{cases}
    2^{-dj}   \quad  &0 < t \le 2^{(j+1)(d-2)}, \\ 
    0          \quad   &2^{(j+1)(d-2)} < t <\infty .
    \end{cases}
\end{equation*}
Hence
\[
  \norm{\chi_{A_j} |\cdot|^{2-d}}_{L^{\frac{d}{d-2},1}(\IB^d)} \le C_d \left (2^{-dj} \right )^{\frac{d-2}{d}}  2^{(j+1)(d-2)} = 2^{d-2}C_d,
\]
which proves that
\[
\sup_{j\in \N_0} \norm{\chi_{A_j} |\cdot|^{2-d}}_{L^{\frac{d}{d-2},1}(\IB^d)} = C_d <\infty,
\]
and finishes the proof of the lemma.
\end{proof}
The previous lemma shows that the set of radial potentials in the Lorentz space $L^{d/2,\infty}(\IB^d) $ are contained in $\cV_d$ with $d>2$. Among other things, in the radial case this implies that all potentials $V(x) = |x|^{-2}f(|x|)$ with $f$ bounded, belong to $\cV_d$. The inclusion $L^{d/2,\infty}(\IB^d) \subset \cV_d$ is strict, since $\cV_d$ also contains any radial $L^1(\IB^d)$ potential which vanishes in a neighborhood of the origin.
It is not clear that this inclusion holds in dimension $d=2$. Indeed, we can not apply Hölder inequality for Lorentz spaces in this case, (since $L^{\infty, q}(\IB^2)= \{0\}$ for $q \not= \infty$). Nevertheless, note that the critical potential $V(x)= c |x|^{-2}$ belongs to $\cV_2 $ (and also to $L^{1,\infty}(\IB^2)$). In principle one can only guarantee in this case the trivial inclusion $L^1(\IB^2)\subset\cV_2$.

\begin{remark}\label{remark:norms_2}
In general,  the Schrödinger equation (\ref{id:calderon_direct}) is not (uniquely) solvable for potentials in $ L^{d/2, \infty}(\IB^d)$. Nevertheless, for potentials $V$ with a {\it{small norm}} living in the  so-called Fefferman-Phong class $F_p \supset L^{\frac{d}{2}, \infty}(\R^d)$ with $ \frac{d-1}{2}<p<\frac{d}{2}$ and $d\ge 3$, it is shown in \cite[Proof of Lemma 2]{Ch90},  that the DtN map $\Lambda_V$ is always well defined : using a Poincaré-type inequality, we see that the bilinear form related to the operator $H= -\Delta+V $ is continuous and coercive in $H_0^1(\IB^d)$. It follows that $0$ is not a Dirichlet eigenvalue of $H$. Moreover, Chanillo shows that the map $V \mapsto \Lambda_V$ is injective. This last result is closely related to the unique continuation principle (UCP). Generically, (UCP) does not hold for potentials belonging to these Lorentz spaces, (see the nice counterexamples in \cite{KT02}), except for potentials with a small norm (\cite{JK85}).
\end{remark}


\bibliographystyle{myalpha}

\bibliography{references_radial_calderon}

\end{document}